\documentclass[final,3p,11pt,authoryear]{elsarticle}

\makeatletter
\def\ps@pprintTitle{%
 \let\@oddhead\@empty
 \let\@evenhead\@empty
 \def\@oddfoot{}%
 \let\@evenfoot\@oddfoot}
\makeatother

\usepackage{amssymb}
\usepackage{amsthm}
\usepackage{lineno}  

\usepackage{graphicx,mathtools,natbib,geometry}
\usepackage[colorlinks]{hyperref}
\makeatletter
\let\Hy@linktoc\Hy@linktoc@none
\makeatother
\usepackage{appendix} 
\usepackage{dsfont} 
\usepackage{mathrsfs} 
\usepackage{caption,subcaption,mwe}
\usepackage{pdflscape} 
\usepackage{listings} 
\usepackage[table]{xcolor} 

\newtheorem{theorem}{Theorem}[section]
\newtheorem{corollary}[theorem]{Corollary}
\newtheorem{proposition}[theorem]{Proposition}

\newtheorem{lemma}[theorem]{Lemma}

\makeatletter \@addtoreset{equation}{section} \makeatother

\makeatletter
\let\c@table\c@figure 
\let\ftype@table\ftype@figure 
\makeatother

\newcommand{\R}{\mathbb{R}}

\newcommand{\EE}{\mathbb{E}}
\newcommand{\BB}{\mathbb{B}\mathrm{ias}}
\newcommand{\VV}{\mathbb{V}\mathrm{ar}}
\newcommand{\nvert}[0]{\, \vert \, }

\newcommand{\OO}{\mathcal O}
\newcommand{\oo}{\mathrm{o}}
\newcommand{\leqdef}{\vcentcolon=}

\newcommand{\rd}{{\rm d}}

\newcommand{\ind}{\mathds{1}}

\newcommand{\e}{\varepsilon}
\newcommand{\ii}{\mathrm{i}}

\begin{document}

\begin{frontmatter}

    \title{A study of seven asymmetric kernels for the estimation of cumulative distribution functions}%

    \author[a1]{Pierre Lafaye de Micheaux}%
    \author[a2]{Fr\'ed\'eric Ouimet\texorpdfstring{\corref{cor1}}{)}}%

    \address[a1]{School of Mathematics and Statistics, UNSW Sydney, Australia.}%
    \address[a2]{PMA Department, California Institute of Technology, Pasadena, USA.}%

    \cortext[cor1]{Corresponding author}%
    \ead{ouimetfr@caltech.edu}%


    \begin{abstract}
        In \cite{Mombeni_et_al_2019_accepted}, Birnbaum-Saunders and Weibull kernel estimators were introduced for the estimation of cumulative distribution functions (c.d.f.s) supported on the half-line $[0,\infty)$.
        They were the first authors to use asymmetric kernels in the context of c.d.f.\ estimation.
        Their estimators were shown to perform better numerically than traditional methods such as the basic kernel method and the boundary modified version from \cite{MR3072469}.
        In the present paper, we complement their study by introducing five new asymmetric kernel c.d.f.\ estimators, namely the Gamma, inverse Gamma, lognormal, inverse Gaussian and reciprocal inverse Gaussian kernel c.d.f.\ estimators.
        For these five new estimators, we prove the asymptotic normality and we find asymptotic expressions for the following quantities: bias, variance, mean squared error and mean integrated squared error. A numerical study then compares the performance of the five new c.d.f.\ estimators against traditional methods and the Birnbaum-Saunders and Weibull kernel c.d.f.\ estimators from \cite{Mombeni_et_al_2019_accepted}. By using the same experimental design, we show that the lognormal and Birnbaum-Saunders kernel c.d.f.\ estimators perform the best overall, while the other asymmetric kernel estimators are sometimes better but always at least competitive against the boundary kernel method.
    \end{abstract}

    \begin{keyword}
        asymmetric kernels \sep asymptotic statistics \sep nonparametric statistics \sep Gamma \sep inverse Gamma \sep lognormal \sep inverse Gaussian \sep reciprocal inverse Gaussian \sep Birnbaum-Saunders \sep Weibull \sep bias \sep variance \sep mean squared error \sep mean integrated squared error \sep asymptotic normality
        \MSC[2010]{Primary: 62G05 Secondary: 60F05, 62G20}
    \end{keyword}

\end{frontmatter}

\vspace{-2mm}
\section{Introduction}\label{sec:intro}

In the context of density estimation, asymmetric kernel estimators were introduced by \cite{doi:10.2307/2347365} on the simplex and studied theoretically for the first time by \cite{MR1718494} on $[0,1]$ (using a Beta kernel), and by \cite{MR1794247} on $[0,\infty)$ (using a Gamma kernel).
These estimators are designed so that the bulk of the kernel function varies with each point $x$ in the support of the target density.
More specifically, the parameters of the kernel function can vary in a way that makes the mode, the median or the mean equal to $x$.
This variable smoothing allows asymmetric kernel estimators to behave better than traditional kernel estimators (see, e.g., \cite{MR79873} and \cite{MR143282}) near the boundary of the support.
Since the variable smoothing is integrated directly in the parametrization of the kernel function, asymmetric kernel estimators are also usually simpler to implement than boundary kernel methods (see, e.g., \cite{doi:10.1007/BFb0098489}, \cite{MR745507}, \cite{MR816088}, \cite{MR1130920} and \cite{MR1649872,MR1752313}).
For these two reasons, asymmetric kernel estimators are, by now, well known solutions to the boundary bias problem from which traditional kernel estimators suffer. In the past twenty years, various asymmetric kernels have been considered in the literature on density estimation:
\vspace{-1.5mm}
\begin{itemize}\setlength\itemsep{-1mm}
    \item Beta, when the target density is supported on $[0,1]$, see, e.g., \cite{MR1718494}, \cite{MR1985506}, \cite{doi:10.1016/j.jbankfin.2003.10.018}, \cite{MR2206532}, \cite{MR2756441}, \cite{MR2568128}, \cite{MR2598955}, \cite{MR2775207}, \cite{MR3333996}, \cite{MR3463548};
    \item Gamma, inverse Gamma, lognormal, inverse Gaussian, reciprocal inverse Gaussian, Birnbaum-Saunders and Weibull, when the target density is supported on $[0,\infty)$, see, e.g.,
        \cite{MR1794247}, \cite{Jin_Kawczak_2003}, \cite{MR2053071}, \cite{MR2179543}, \cite{MR2206532}, \cite{MR2454617,MR2568128,MR2756423}, \cite{MR3131281,MR3843043}, \cite{doi:10.2139/ssrn.2514882}, \cite{MR3540109}, \cite{MR3456321}, \cite{MR3648359}, \cite{MR3713468}, \cite{MR3819800}, \cite{MR2595129}, \cite{MR4096263};
    \item Dirichlet, when the target density is supported on the $d$-dimensional unit simplex, see \cite{doi:10.2307/2347365} and the first theoretical study in \cite{arXiv:2002.06956}. \vspace{-6.5mm}
\end{itemize}
The interested reader is referred to \cite{MR3821525} and Section 2 in \cite{arXiv:2002.06956} for a review of some of these papers and an extensive list of papers dealing with asymmetric kernels in other settings.

\vspace{2mm}
In contrast, there are almost no papers dealing with the estimation of cumulative distribution functions (c.d.f.s) in the literature on asymmetric kernels.
In fact, to the best of our knowledge, \cite{Mombeni_et_al_2019_accepted} seems to be the first (and only) paper in this direction if we exclude the closely related theory of Bernstein estimators.%
\footnote{In the setting of Bernstein estimators, c.d.f.\ estimation on compact sets was tackled for example by \cite{MR1910059}, \cite{MR2488150}, \cite{MR2960952}, \cite{MR2925964}, \cite{MR3473628}, \cite{MR3740720}, \cite{MR3899096} and \cite{MR3950592} in the univariate setting, and by \cite{MR2270097}, \cite{MR3474765}, \cite{doi:10.1080/03610926.2020.1734832} and \cite{arXiv:2002.07758,arXiv:2006.11756} in the multivariate setting. In \cite{arXiv:2005.09994}, the authors introduced Bernstein estimators with Poisson weights (also called Szasz estimators) for the estimation of c.d.f.s that are supported on $[0,\infty)$, see also \cite{arXiv:2010.05146}.}

\vspace{2mm}
In the present paper, we complement the study in \cite{Mombeni_et_al_2019_accepted} by introducing five new asymmetric kernel c.d.f.\ estimators, namely the Gamma, inverse Gamma, lognormal, inverse Gaussian and reciprocal inverse Gaussian kernel c.d.f.\ estimators.
Our goal is to prove several asymptotic properties for these five new estimators (bias, variance, mean squared error, mean integrated squared error and asymptotic normality) and compare their numerical performance against traditional methods and against the Birnbaum-Saunders and Weibull kernel c.d.f.\ estimators from \cite{Mombeni_et_al_2019_accepted}.
As we will see in the discussion of the results (Section~\ref{sec:discussion}), the lognormal and Birnbaum-Saunders kernel c.d.f.\ estimators perform the best overall, while the other asymmetric kernel estimators are sometimes better but always at least competitive against the boundary kernel method from \cite{MR3072469}.

\section{The models}\label{sec:models}

Let $X_1, X_2, \dots, X_n$ be a sequence of i.i.d.\ observations from an unknown cumulative distribution function $F$ supported on the half-line $[0,\infty)$.
We consider the following seven asymmetric kernel estimators (the first five are new):
\begin{align}
    &\hat{F}_{n,b}^{\mathrm{Gam}}(x) \leqdef \frac{1}{n} \sum_{i=1}^n \overline{K}_{\mathrm{Gam}}(X_i \nvert b^{-1} x + 1, b), \label{def:G.kernel.estimator} \\[-0.5mm]
    &\hat{F}_{n,b}^{\mathrm{IGam}}(x) \leqdef \frac{1}{n} \sum_{i=1}^n \overline{K}_{\mathrm{IGam}}(X_i \nvert b^{-1} + 1, x^{-1} b), \label{def:IGam.kernel.estimator} \\[-0.5mm]
    &\hat{F}_{n,b}^{\mathrm{LN}}(x) \leqdef \frac{1}{n} \sum_{i=1}^n \overline{K}_{\mathrm{LN}}(X_i \nvert \log x, \sqrt{b}), \label{def:LN.kernel.estimator} \\[-0.5mm]
    &\hat{F}_{n,b}^{\mathrm{IGau}}(x) \leqdef \frac{1}{n} \sum_{i=1}^n \overline{K}_{\mathrm{IGau}}(X_i \nvert x, b^{-1} x), \label{def:IGau.kernel.estimator} \\[-0.5mm]
    &\hat{F}_{n,b}^{\mathrm{RIG}}(x) \leqdef \frac{1}{n} \sum_{i=1}^n \overline{K}_{\mathrm{RIG}}(X_i \nvert x^{-1} (1 - b)^{-1}, x^{-1} b^{-1}), \label{def:RIG.kernel.estimator} \\[-0.5mm]
    &\hat{F}_{n,b}^{\mathrm{B-S}}(x) \leqdef \frac{1}{n} \sum_{i=1}^n \overline{K}_{\mathrm{B-S}}(X_i \nvert x, \sqrt{b}), \label{def:BS.kernel.estimator} \\[-0.5mm]
    &\hat{F}_{n,b}^{\mathrm{W}}(x) \leqdef \frac{1}{n} \sum_{i=1}^n \overline{K}_{\mathrm{W}}(X_i \nvert x / \Gamma(1 + b), b^{-1}), \label{def:W.kernel.estimator}
\end{align}
where $b > 0$ is a {\it smoothing} (or {\it bandwidth}) parameter, and
\begin{alignat*}{3}
    &\overline{K}_{\mathrm{Gam}}(t \nvert \alpha, \theta) \leqdef \frac{\Gamma(\alpha, t / \theta)}{\Gamma(\alpha)}, \quad && \alpha, \theta > 0, \\[-0.5mm]
    &\overline{K}_{\mathrm{IGam}}(t \nvert \alpha, \theta) \leqdef 1 - \frac{\Gamma(\alpha, 1 / (t \theta))}{\Gamma(\alpha)}, \quad && \alpha, \theta > 0, \\[-0.5mm]
    &\overline{K}_{\mathrm{LN}}(t \nvert \mu, \sigma) \leqdef 1 - \Phi\bigg(\frac{\log t - \mu}{\sigma}\bigg), \quad && \mu, \sigma > 0, \\[-0.5mm]
    &\overline{K}_{\mathrm{IGau}}(t \nvert \mu, \lambda) \leqdef 1 - \Phi\bigg(\sqrt{\frac{\lambda}{t}} \Big(\frac{t}{\mu} - 1\Big)\bigg) - e^{2\lambda / \mu} \Phi\bigg(-\sqrt{\frac{\lambda}{t}} \Big(\frac{t}{\mu} + 1\Big)\bigg), \quad && \mu, \lambda > 0, \\[0.5mm]
    &\overline{K}_{\mathrm{RIG}}(t \nvert \mu, \lambda) \leqdef \Phi\bigg(\sqrt{\lambda t} \Big(\frac{1}{t \mu} - 1\Big)\bigg) + e^{2\lambda / \mu} \Phi\bigg(-\sqrt{\lambda t} \Big(\frac{1}{t \mu} + 1\Big)\bigg), \quad && \mu, \lambda > 0, \\[-0.5mm]
    &\overline{K}_{\mathrm{B-S}}(t \nvert \beta, \alpha) \leqdef 1 - \Phi\bigg(\frac{1}{\alpha} \Big(\sqrt{\frac{t}{\beta}} - \sqrt{\frac{\beta}{t}}\Big)\bigg), \quad && \beta, \alpha > 0, \\[-0.5mm]
    &\overline{K}_{\mathrm{W}}(t \nvert \lambda, k) \leqdef \exp\bigg(-\Big(\frac{t}{\lambda}\Big)^k\bigg), \quad && \lambda, k > 0,
\end{alignat*}
denote, respectively, the survival function of the
\begin{itemize}\setlength\itemsep{0mm}
    \item $\text{Gamma}\hspace{0.2mm}(\alpha,\theta)$ distribution (with shape/scale parametrization),
    \item $\text{InverseGamma}\hspace{0.2mm}(\alpha,\theta)$ distribution (with shape/scale parametrization),
    \item $\text{LogNormal}\hspace{0.2mm}(\mu,\sigma)$ distribution,
    \item $\text{InverseGaussian}\hspace{0.2mm}(\mu,\lambda)$ distribution,
    \item $\text{ReciprocalInverseGaussian}\hspace{0.2mm}(\mu,\lambda)$ distribution,
    \item $\text{Birnbaum-Saunders}\hspace{0.2mm}(\beta,\alpha)$ distribution,
    \item $\text{Weibull}\hspace{0.2mm}(\lambda,k)$ distribution.
\end{itemize}
The function $\Gamma(\alpha,z) \leqdef \int_z^{\infty} t^{\alpha - 1} e^{-t} \rd t$ denotes the upper incomplete gamma function (where $\Gamma(\alpha) \leqdef \Gamma(\alpha,0)$), and $\Phi$ denotes the c.d.f.\ of the standard normal distribution.
The parametrizations are chosen so that
\begin{itemize}\setlength\itemsep{0mm}
    \item The mode of the kernel function in \eqref{def:G.kernel.estimator} is $x$;
    \item The median of the kernel function in \eqref{def:LN.kernel.estimator} and \eqref{def:BS.kernel.estimator} is $x$;
    \item The mean of the kernel function in \eqref{def:IGam.kernel.estimator}, \eqref{def:IGau.kernel.estimator}, \eqref{def:RIG.kernel.estimator} and \eqref{def:W.kernel.estimator} is $x$.
\end{itemize}

\vspace{2mm}
In this paper, we will compare the numerical performance of the above seven asymmetric kernel c.d.f.\ estimators against the following three traditional estimators ($K$ here is the c.d.f.\ of a kernel function):
\begin{align}
    &\hat{F}_{n,b}^{\mathrm{OK}}(x) \leqdef \frac{1}{n} \sum_{i=1}^n K\Big(\frac{x - X_i}{b}\Big), \label{def:OK.estimator} \\
    &\hat{F}_{n,b}^{\mathrm{BK}}(x) \leqdef \frac{1}{n} \sum_{i=1}^n \bigg\{K\Big(\frac{x - X_i}{b}\Big) \ind_{[b,\infty)}(x) + K\Big(\frac{x - X_i}{x}\Big) \ind_{(0,b)}(x)\bigg\}, \label{def:BK.estimator} \\
    &\hat{F}_n^{\mathrm{EDF}}(x) \leqdef \frac{1}{n} \sum_{i=1}^n \ind_{[X_i,\infty)}(x). \label{def:EDF.estimator}
\end{align}
which denote, respectively, the ordinary kernel (OK) c.d.f.\ estimator (from \cite{Tiago_1963}, \cite{MR0166862} or \cite{MR184336}), the boundary kernel (BK) c.d.f.\ estimator (from Example 2.3 in \cite{MR3072469}) and the empirical c.d.f.\ (EDF) estimator.

\section{Outline, assumptions and notation}\label{sec:outline.assumptions.notation}

    \subsection{Outline}

    In Section~\ref{sec:G.kernel.results},~\ref{sec:IGam.kernel.results},~\ref{sec:LN.kernel.results},~\ref{sec:IGau.kernel.results} and~\ref{sec:RIG.kernel.results}, the asymptotic normality, and the asymptotic expressions for the bias, variance, mean squared error (MSE) and mean integrated squared error (MISE), are stated for the $\mathrm{Gam}$, $\mathrm{IGam}$, $\mathrm{LN}$, $\mathrm{IGau}$ and $\mathrm{RIG}$ kernel c.d.f.\ estimators, respectively.
    The proofs can be found in Appendix~\ref{sec:proof.results.G.kernel},~\ref{sec:proof.results.IGam.kernel},~\ref{sec:proof.results.LN.kernel},~\ref{sec:proof.results.IGau.kernel} and~\ref{sec:proof.results.RIG.kernel}, respectively.
    Aside from the asymptotic normality (which can easily be deduced), these results were obtained for the Birnbaum-Saunders and Weibull kernel c.d.f.\ estimators in \cite{Mombeni_et_al_2019_accepted}.
    In Section~\ref{sec:numerical.study}, we compare the performance of all seven asymmetric kernel estimators above with the three traditional estimators $\mathrm{OK}$, $\mathrm{BK}$ and $\mathrm{EDF}$, defined in \eqref{def:OK.estimator}, \eqref{def:BK.estimator} and \eqref{def:EDF.estimator}.
    A discussion of the results and our conclusion follow in Section~\ref{sec:discussion} and Section~\ref{sec:conclusion}.
    Technical calculations for the proofs of the asymptotic results are gathered in Appendix~\ref{sec:tech.lemmas}.

    \subsection{Assumptions}

    Throughout the paper, we make the following two basic assumptions:
    \begin{enumerate}
        \item The target c.d.f.\ $F$ has two continuous and bounded derivatives;
        \item The {\it smoothing} (or {\it bandwidth}) parameter $b = b_n > 0$ satisfies $b\to 0$ as $n\to \infty$.
    \end{enumerate}

    \subsection{Notation}

    Throughout the paper, the notation $u = \OO(v)$ means that $\limsup |u / v| < C < \infty$ as $n\to \infty$.
    The positive constant $C$ can depend on the target c.d.f.\ $F$, but no other variable unless explicitly written as a subscript.
    For example, if $C$ depends on a given point $x\in (0,\infty)$, we would write $u = \OO_x(v)$.
    Similarly, the notation $u = \oo(v)$ means that $\lim |u / v| = 0$ as $n\to \infty$.
    The same rule applies for the subscript.
    The expression `$\stackrel{\scriptscriptstyle \mathscr{D}}{\longrightarrow}$' will denote the convergence in law (or distribution).

\section{Asymptotic properties of the c.d.f.\ estimator with Gam kernel}\label{sec:G.kernel.results}

In this section, we find the asymptotic properties of the Gamma ($\mathrm{Gam}$) kernel estimator defined in \eqref{def:G.kernel.estimator}.

\begin{lemma}[{\color{blue} Bias and variance}]\label{lem:bias.variance.G.kernel}
    For any given $x\in (0,\infty)$,
    \begin{align}
        &\BB[\hat{F}_{n,b}^{\mathrm{Gam}}(x)] \leqdef \EE[\hat{F}_{n,b}^{\mathrm{Gam}}(x)] - F(x) = b \cdot (f(x) + \frac{x}{2} f'(x)) + \oo_x(b), \\
        &\VV(\hat{F}_{n,b}^{\mathrm{Gam}}(x)) = n^{-1} F(x) (1 - F(x)) - n^{-1} b^{1/2} \cdot \frac{\sqrt{x} f(x)}{\sqrt{\pi}} + \OO_x(n^{-1} b).
    \end{align}
\end{lemma}

\begin{corollary}[{\color{blue} Mean squared error}]\label{cor:MSE.G.kernel}
    For any given $x\in (0,\infty)$,
    \begin{align}
        \mathrm{MSE}(\hat{F}_{n,b}^{\mathrm{Gam}}(x))
        &= \VV(\hat{F}_{n,b}^{\mathrm{Gam}}(x)) + \big(\BB[\hat{F}_{n,b}^{\mathrm{Gam}}(x)]\big)^2 \notag \\
        &= n^{-1} F(x) (1 - F(x)) - n^{-1} b^{1/2} \cdot \frac{\sqrt{x} f(x)}{\sqrt{\pi}} \notag \\[-1mm]
        &\quad+ b^2 \cdot (f(x) + \frac{x}{2} f'(x))^2 + \OO_x(n^{-1} b) + \oo_x(b^2).
    \end{align}
    In particular, if $f(x) \cdot (f(x) + x f'(x)) \neq 0$, the asymptotically optimal choice of $b$, with respect to $\mathrm{MSE}$, is
    \begin{equation}
        b_{\mathrm{opt}} = n^{-2/3} \left[\frac{4 \cdot (f(x) + \frac{x}{2} f'(x))^2}{\sqrt{x} f(x) / \sqrt{\pi}}\right]^{-2/3}
    \end{equation}
    with
    \begin{equation}
        \begin{aligned}
            \mathrm{MSE}(\hat{F}_{n,b_{\mathrm{opt}}}^{\mathrm{Gam}}(x))
            &= n^{-1} F(x) (1 - F(x)) \\[-1.5mm]
            &\quad- n^{-4/3} \, \frac{3}{4} \left[\frac{(\sqrt{x} f(x) / \sqrt{\pi})^4}{4 \cdot (f(x) + \frac{x}{2} f'(x))^2}\right]^{1/3} \hspace{-2mm} + \oo_x(n^{-4/3}).
        \end{aligned}
    \end{equation}
\end{corollary}

\begin{proposition}[{\color{blue} Mean integrated squared error}]\label{prop:MISE.G.kernel}
    Assuming that the target density $f = F'$ satisfies
    \begin{equation}
        \int_0^{\infty} \sqrt{x} f(x) \rd x < \infty \quad \text{and} \quad \int_0^{\infty} (f(x) + \frac{x}{2} f'(x))^2 \rd x < \infty,
    \end{equation}
    then we have
    \begin{align}
        \mathrm{MISE}(\hat{F}_{n,b}^{\mathrm{Gam}})
        &= \int_0^{\infty} \VV(\hat{F}_{n,b}^{\mathrm{Gam}}(x)) \rd x + \int_0^{\infty} \big(\BB[\hat{F}_{n,b}^{\mathrm{Gam}}(x)]\big)^2 \rd x \notag \\
        &= n^{-1} \int_0^{\infty} F(x) (1 - F(x)) \rd x - n^{-1} b^{1/2} \int_0^{\infty} \frac{\sqrt{x} f(x)}{\sqrt{\pi}} \rd x \notag \\[-1mm]
        &\quad+ b^2 \int_0^{\infty} (f(x) + \frac{x}{2} f'(x))^2 \rd x + \oo(n^{-1} b^{1/2}) + \oo(b^2).
    \end{align}
    In particular, if $f(x) \cdot (f(x) + x f'(x)) \neq 0$, the asymptotically optimal choice of $b$, with respect to $\mathrm{MISE}$, is
    \vspace{-3mm}
    \begin{equation}\label{eq:b.opt.MISE.G}
        b_{\mathrm{opt}} = n^{-2/3} \left[\frac{4 \int_0^{\infty} (f(x) + \frac{x}{2} f'(x))^2 \rd x}{\int_0^{\infty} \sqrt{x} f(x) / \sqrt{\pi} \rd x}\right]^{-2/3}
    \end{equation}
    with
    \begin{equation}
        \begin{aligned}
            \mathrm{MISE}(\hat{F}_{n,b_{\mathrm{opt}}}^{\mathrm{Gam}})
            &= n^{-1} \int_0^{\infty} F(x) (1 - F(x)) \rd x \\[-1.5mm]
            &\quad- n^{-4/3} \, \frac{3}{4} \left[\frac{\big(\int_0^{\infty} \sqrt{x} f(x) / \sqrt{\pi} \rd x\big)^4}{4 \int_0^{\infty} (f(x) + \frac{x}{2} f'(x))^2 \rd x}\right]^{1/3} \\
            &\quad+ \oo(n^{-4/3}).
        \end{aligned}
    \end{equation}
\end{proposition}

\begin{proposition}[{\color{blue} Asymptotic normality}]\label{prop:asymptotic.normality.G.kernel}
    For any $x > 0$ such that $0 < F(x) < 1$, we have the following convergence in distribution:
    \begin{equation}
        n^{1/2} (\hat{F}_{n,b}^{\mathrm{Gam}}(x) - \EE[\hat{F}_{n,b}^{\mathrm{Gam}}(x)]) \stackrel{\mathscr{D}}{\longrightarrow} \mathcal{N}(0, \sigma^2(x)), \quad \text{as } b\to 0, \, n\to \infty,
    \end{equation}
    where $\sigma^2(x) \leqdef F(x) (1 - F(x))$.
    In particular, Lemma~\ref{lem:bias.variance.G.kernel} implies
    \begin{alignat}{3}
        &n^{1/2} (\hat{F}_{n,b}^{\mathrm{Gam}}(x) - F(x)) \stackrel{\mathscr{D}}{\longrightarrow} \mathcal{N}(0, \sigma^2(x)), \quad &&\text{if } n^{1/2} b \to 0, \\
        &n^{1/2} (\hat{F}_{n,b}^{\mathrm{Gam}}(x) - F(x)) \stackrel{\mathscr{D}}{\longrightarrow} \mathcal{N}(\lambda \cdot (f(x) + \tfrac{x}{2} f'(x)), \sigma^2(x)), \quad &&\text{if } n^{1/2} b \to \lambda,
    \end{alignat}
    for any constant $\lambda > 0$.
\end{proposition}

\section{Asymptotic properties of the c.d.f.\ estimator with IGam kernel}\label{sec:IGam.kernel.results}

In this section, we find the asymptotic properties of the inverse Gamma ($\mathrm{IGam}$) kernel estimator defined in \eqref{def:IGam.kernel.estimator}.

\begin{lemma}[{\color{blue} Bias and variance}]\label{lem:bias.variance.IGam.kernel}
    For any given $x\in (0,\infty)$,
    \begin{align}
        \BB[\hat{F}_{n,b}^{\mathrm{IGam}}(x)]
        &\leqdef \EE[\hat{F}_{n,b}^{\mathrm{IGam}}(x)] - F(x) = b \cdot \frac{x^2}{2} f'(x) + \oo_x(b), \\[1mm]
        \VV(\hat{F}_{n,b}^{\mathrm{IGam}}(x))
        &= n^{-1} F(x) (1 - F(x)) - n^{-1} b^{1/2} \cdot \frac{x f(x)}{\sqrt{\pi}} + \OO_x(n^{-1} b).
    \end{align}
\end{lemma}

\begin{corollary}[{\color{blue} Mean squared error}]\label{cor:MSE.IGam.kernel}
    For any given $x\in (0,\infty)$,
    \begin{align}
        \mathrm{MSE}(\hat{F}_{n,b}^{\mathrm{IGam}}(x))
        &= \VV(\hat{F}_{n,b}^{\mathrm{IGam}}(x)) + \big(\BB[\hat{F}_{n,b}^{\mathrm{IGam}}(x)]\big)^2 \notag \\
        &= n^{-1} F(x) (1 - F(x)) - n^{-1} b^{1/2} \cdot \frac{x f(x)}{\sqrt{\pi}} \notag \\[-1mm]
        &\quad+ b^2 \cdot \frac{x^4}{4} (f'(x))^2 + \OO_x(n^{-1} b) + \oo_x(b^2).
    \end{align}
    In particular, if $f(x) \cdot f'(x) \neq 0$, the asymptotically optimal choice of $b$, with respect to $\mathrm{MSE}$, is
    \begin{equation}
        b_{\mathrm{opt}} = n^{-2/3} \left[\frac{4 \cdot \frac{x^4}{4} (f'(x))^2}{x f(x) / \sqrt{\pi}}\right]^{-2/3}
    \end{equation}
    with
    \begin{equation}
        \begin{aligned}
            \mathrm{MSE}(\hat{F}_{n,b_{\mathrm{opt}}}^{\mathrm{IGam}}(x))
            &= n^{-1} F(x) (1 - F(x)) \\[-1.5mm]
            &\quad- n^{-4/3} \, \frac{3}{4} \left[\frac{\big(x f(x) / \sqrt{\pi}\big)^4}{4 \cdot \frac{x^4}{4} (f'(x))^2}\right]^{1/3} \hspace{-2mm} + \oo_x(n^{-4/3}).
        \end{aligned}
    \end{equation}
\end{corollary}

\begin{proposition}[{\color{blue} Mean integrated squared error}]\label{prop:MISE.IGam.kernel}
    Assuming that the target density $f = F'$ satisfies
    \begin{equation}
        \int_0^{\infty} x f(x) \rd x < \infty \quad \text{and} \quad \int_0^{\infty} x^4 (f'(x))^2 \rd x < \infty,
    \end{equation}
    then we have
    \begin{align}
        \mathrm{MISE}(\hat{F}_{n,b}^{\mathrm{IGam}})
        &= \int_0^{\infty} \VV(\hat{F}_{n,b}^{\mathrm{IGam}}(x)) \rd x + \int_0^{\infty} \big(\BB[\hat{F}_{n,b}^{\mathrm{IGam}}(x)]\big)^2 \rd x \notag \\
        &= n^{-1} \int_0^{\infty} F(x) (1 - F(x)) \rd x - n^{-1} b^{1/2} \int_0^{\infty} \frac{x f(x)}{\sqrt{\pi}} \rd x \notag \\
        &\quad+ b^2 \int_0^{\infty} \frac{x^4}{4} (f'(x))^2 \rd x + \oo(n^{-1} b^{1/2}) + \oo(b^2).
    \end{align}
    In particular, if $\int_0^{\infty} x^4(f'(x))^2 \rd x > 0$, the asymptotically optimal choice of $b$, with respect to $\mathrm{MISE}$, is
    \begin{equation}\label{eq:b.opt.MISE.IGam}
        b_{\mathrm{opt}} = n^{-2/3} \left[\frac{4 \int_0^{\infty} \frac{x^4}{4} (f'(x))^2 \rd x}{\int_0^{\infty} x f(x) / \sqrt{\pi} \rd x}\right]^{-2/3}
    \end{equation}
    with
    \begin{equation}
        \begin{aligned}
            \mathrm{MISE}(\hat{F}_{n,b_{\mathrm{opt}}}^{\mathrm{IGam}})
            &= n^{-1} \int_0^{\infty} F(x) (1 - F(x)) \rd x \\[-1.5mm]
            &\quad- n^{-4/3} \, \frac{3}{4} \left[\frac{\big(\int_0^{\infty} x f(x) / \sqrt{\pi} \rd x\big)^4}{4 \int_0^{\infty} \frac{x^4}{4} (f'(x))^2 \rd x}\right]^{1/3} \\
            &\quad+ \oo(n^{-4/3}).
        \end{aligned}
    \end{equation}
\end{proposition}

\begin{proposition}[{\color{blue} Asymptotic normality}]\label{prop:asymptotic.normality.IGam.kernel}
    For any $x > 0$ such that $0 < F(x) < 1$, we have the following convergence in distribution:
    \begin{equation}
        n^{1/2} (\hat{F}_{n,b}^{\mathrm{IGam}}(x) - \EE[\hat{F}_{n,b}^{\mathrm{IGam}}(x)]) \stackrel{\mathscr{D}}{\longrightarrow} \mathcal{N}(0, \sigma^2(x)), \quad \text{as } b\to 0, \, n\to \infty,
    \end{equation}
    where $\sigma^2(x) \leqdef F(x) (1 - F(x))$.
    In particular, Lemma~\ref{lem:bias.variance.IGam.kernel} implies
    \begin{alignat}{3}
        &n^{1/2} (\hat{F}_{n,b}^{\mathrm{IGam}}(x) - F(x)) \stackrel{\mathscr{D}}{\longrightarrow} \mathcal{N}(0,\sigma^2(x)), \quad &&\text{if } n^{1/2} b \to 0, \\
        &n^{1/2} (\hat{F}_{n,b}^{\mathrm{IGam}}(x) - F(x)) \stackrel{\mathscr{D}}{\longrightarrow} \mathcal{N}(\lambda \cdot \tfrac{x^2 f'(x)}{2}, \sigma^2(x)), \quad &&\text{if } n^{1/2} b \to \lambda,
    \end{alignat}
    for any constant $\lambda > 0$.
\end{proposition}

\section{Asymptotic properties of the c.d.f.\ estimator with LN kernel}\label{sec:LN.kernel.results}

In this section, we find the asymptotic properties of the LogNormal (LN) kernel estimator defined in \eqref{def:LN.kernel.estimator}.

\begin{lemma}[{\color{blue} Bias and variance}]\label{lem:bias.variance.LN.kernel}
    For any given $x\in (0,\infty)$,
    \begin{align}
        &\BB[\hat{F}_{n,b}^{\mathrm{LN}}(x)] \leqdef \EE[\hat{F}_{n,b}^{\mathrm{LN}}(x)] - F(x) = b \cdot \frac{x}{2} (f(x) + x f'(x)) + \oo_x(b), \\
        &\VV(\hat{F}_{n,b}^{\mathrm{LN}}(x)) = n^{-1} F(x) (1 - F(x)) - n^{-1} b^{1/2} \cdot \frac{x f(x)}{\sqrt{\pi}} + \OO_x(n^{-1} b).
    \end{align}
\end{lemma}

\begin{corollary}[{\color{blue} Mean squared error}]\label{cor:MSE.LN.kernel}
    For any given $x\in (0,\infty)$,
    \begin{align}
        \mathrm{MSE}(\hat{F}_{n,b}^{\mathrm{LN}}(x))
        &= \VV(\hat{F}_{n,b}^{\mathrm{LN}}(x)) + \big(\BB[\hat{F}_{n,b}^{\mathrm{LN}}(x)]\big)^2 \notag \\
        &= n^{-1} F(x) (1 - F(x)) - n^{-1} b^{1/2} \cdot \frac{x f(x)}{\sqrt{\pi}} \notag \\[-1mm]
        &\quad+ b^2 \cdot \frac{x^2}{4} (f(x) + x f'(x))^2 + \OO_x(n^{-1} b) + \oo_x(b^2).
    \end{align}
    In particular, if $f(x) \cdot (f(x) + x f'(x)) \neq 0$, the asymptotically optimal choice of $b$, with respect to $\mathrm{MSE}$, is
    \begin{equation}
        b_{\mathrm{opt}} = n^{-2/3} \left[\frac{4 \cdot \frac{x^2}{4} (f(x) + x f'(x))^2}{x f(x) / \sqrt{\pi}}\right]^{-2/3}
    \end{equation}
    with
    \begin{equation}
        \begin{aligned}
            \mathrm{MSE}(\hat{F}_{n,b_{\mathrm{opt}}}^{\mathrm{LN}}(x))
            &= n^{-1} F(x) (1 - F(x)) \\[-1.5mm]
            &\quad- n^{-4/3} \, \frac{3}{4} \left[\frac{(x f(x) / \sqrt{\pi})^4}{4 \cdot \frac{x^2}{4} (f(x) + x f'(x))^2}\right]^{1/3} \hspace{-2mm} + \oo_x(n^{-4/3}).
        \end{aligned}
    \end{equation}
\end{corollary}

\begin{proposition}[{\color{blue} Mean integrated squared error}]\label{prop:MISE.LN.kernel}
    Assuming that the target density $f = F'$ satisfies
    \begin{equation}
        \int_0^{\infty} x f(x) \rd x < \infty \quad \text{and} \quad \int_0^{\infty} x^2 (f(x) + x f'(x))^2 \rd x < \infty,
    \end{equation}
    then we have
    \begin{align}
        \mathrm{MISE}(\hat{F}_{n,b}^{\mathrm{LN}})
        &= \int_0^{\infty} \VV(\hat{F}_{n,b}^{\mathrm{LN}}(x)) \rd x + \int_0^{\infty} \big(\BB[\hat{F}_{n,b}^{\mathrm{LN}}(x)]\big)^2 \rd x \notag \\
        &= n^{-1} \int_0^{\infty} F(x) (1 - F(x)) \rd x - n^{-1} b^{1/2} \int_0^{\infty} \frac{x f(x)}{\sqrt{\pi}} \rd x \notag \\[-1mm]
        &\quad+ b^2 \int_0^{\infty} \frac{x^2}{4} (f(x) + x f'(x))^2 \rd x + \oo(n^{-1} b^{1/2}) + \oo(b^2).
    \end{align}
    In particular, if $f(x) \cdot (f(x) + x f'(x)) \neq 0$, the asymptotically optimal choice of $b$, with respect to $\mathrm{MISE}$, is
    \vspace{-3mm}
    \begin{equation}\label{eq:b.opt.MISE.LN}
        b_{\mathrm{opt}} = n^{-2/3} \left[\frac{4 \int_0^{\infty} \frac{x^2}{4} (f(x) + x f'(x))^2 \rd x}{\int_0^{\infty} x f(x) / \sqrt{\pi} \rd x}\right]^{-2/3}
    \end{equation}
    with
    \begin{equation}
        \begin{aligned}
            \mathrm{MISE}(\hat{F}_{n,b_{\mathrm{opt}}}^{\mathrm{LN}})
            &= n^{-1} \int_0^{\infty} F(x) (1 - F(x)) \rd x \\[-1.5mm]
            &\quad- n^{-4/3} \, \frac{3}{4} \left[\frac{\big(\int_0^{\infty} x f(x) / \sqrt{\pi} \rd x\big)^4}{4 \int_0^{\infty} \frac{x^2}{4} (f(x) + x f'(x))^2 \rd x}\right]^{1/3} \\
            &\quad+ \oo(n^{-4/3}).
        \end{aligned}
    \end{equation}
\end{proposition}

\begin{proposition}[{\color{blue} Asymptotic normality}]\label{prop:asymptotic.normality.LN.kernel}
    For any $x > 0$ such that $0 < F(x) < 1$, we have the following convergence in distribution:
    \begin{equation}
        n^{1/2} (\hat{F}_{n,b}^{\mathrm{LN}}(x) - \EE[\hat{F}_{n,b}^{\mathrm{LN}}(x)]) \stackrel{\mathscr{D}}{\longrightarrow} \mathcal{N}(0, \sigma^2(x)), \quad \text{as } b\to 0, \, n\to \infty,
    \end{equation}
    where $\sigma^2(x) \leqdef F(x) (1 - F(x))$.
    In particular, Lemma~\ref{lem:bias.variance.LN.kernel} implies
    \begin{alignat}{3}
        &n^{1/2} (\hat{F}_{n,b}^{\mathrm{LN}}(x) - F(x)) \stackrel{\mathscr{D}}{\longrightarrow} \mathcal{N}(0, \sigma^2(x)), \quad &&\text{if } n^{1/2} b \to 0, \\
        &n^{1/2} (\hat{F}_{n,b}^{\mathrm{LN}}(x) - F(x)) \stackrel{\mathscr{D}}{\longrightarrow} \mathcal{N}(\lambda \cdot \frac{x}{2} (f(x) + x f'(x)), \sigma^2(x)), \quad &&\text{if } n^{1/2} b \to \lambda,
    \end{alignat}
    for any constant $\lambda > 0$.
\end{proposition}

\section{Asymptotic properties of the c.d.f.\ estimator with IGau kernel}\label{sec:IGau.kernel.results}

In this section, we find the asymptotic properties of the inverse Gaussian (IGau) kernel estimator defined in \eqref{def:IGau.kernel.estimator}.

\begin{lemma}[{\color{blue} Bias and variance}]\label{lem:bias.variance.IGau.kernel}
    For any given $x\in (0,\infty)$,
    \begin{align}
        \BB[\hat{F}_{n,b}^{\mathrm{IGau}}(x)]
        &\leqdef \EE[\hat{F}_{n,b}^{\mathrm{IGau}}(x)] - F(x) = b \cdot \frac{x^2}{2} f'(x) + \oo_x(b), \\[1mm]
        \VV(\hat{F}_{n,b}^{\mathrm{IGau}}(x))
        &= n^{-1} F(x) (1 - F(x)) \notag \\
        &\quad- n^{-1} b^{1/2} \cdot \frac{f(x)}{2} \Big[\lim_{b\to 0} b^{-1/2} \, \EE[|T_1 - T_2|]\Big] \notag \\
        &\quad+ \OO_x(n^{-1} b),
    \end{align}
    where $T_1,T_2 \stackrel{\mathrm{i.i.d.}}{\sim} \mathrm{IGau}(x, b^{-1} x)$.
\end{lemma}

\begin{corollary}[{\color{blue} Mean squared error}]\label{cor:MSE.IGau.kernel}
    For any given $x\in (0,\infty)$,
    \begin{align}
        \mathrm{MSE}(\hat{F}_{n,b}^{\mathrm{IGau}}(x))
        &= \VV(\hat{F}_{n,b}^{\mathrm{IGau}}(x)) + \big(\BB[\hat{F}_{n,b}^{\mathrm{IGau}}(x)]\big)^2 \notag \\
        &= n^{-1} F(x) (1 - F(x)) - n^{-1} b^{1/2} \cdot \frac{f(x)}{2} \Big[\lim_{b\to 0} b^{-1/2} \, \EE[|T_1 - T_2|]\Big] \notag \\[-1mm]
        &\quad+ b^2 \cdot \frac{x^4}{4} (f'(x))^2 + \OO_x(n^{-1} b) + \oo_x(b^2),
    \end{align}
    where $T_1,T_2 \stackrel{\mathrm{i.i.d.}}{\sim} \mathrm{IGau}(x, b^{-1} x)$.
    The quantity $\lim_{b\to 0} b^{-1/2} \, \EE[|T_1 - T_2|]$ needs to be approximated numerically.
    In particular, if $f(x) \big[\lim_{b\to 0} b^{-1/2} \, \EE[|T_1 - T_2|]\big] \cdot f'(x) \neq 0$, the asymptotically optimal choice of $b$, with respect to $\mathrm{MSE}$, is
    \begin{equation}
        b_{\mathrm{opt}} = n^{-2/3} \left[\frac{4 \cdot \frac{x^4}{4} (f'(x))^2}{\frac{f(x)}{2} \lim_{b\to 0} b^{-1/2} \, \EE[|T_1 - T_2|]}\right]^{-2/3}
    \end{equation}
    with
    \begin{equation}
        \begin{aligned}
            \mathrm{MSE}(\hat{F}_{n,b_{\mathrm{opt}}}^{\mathrm{IGau}}(x))
            &= n^{-1} F(x) (1 - F(x)) \\[-1.5mm]
            &\quad- n^{-4/3} \, \frac{3}{4} \left[\frac{\big(\frac{f(x)}{2} \lim_{b\to 0} b^{-1/2} \, \EE[|T_1 - T_2|]\big)^4}{4 \cdot \frac{x^4}{4} (f'(x))^2}\right]^{1/3} \\
            &\quad+ \oo_x(n^{-4/3}).
        \end{aligned}
    \end{equation}
\end{corollary}

\begin{proposition}[{\color{blue} Mean integrated squared error}]\label{prop:MISE.IGau.kernel}
    Assuming that the target density $f = F'$ satisfies
    \begin{equation}
        \int_0^{\infty} f(x) \Big[\lim_{b\to 0} b^{-1/2} \, \EE[|T_1 - T_2|]\Big] \rd x < \infty \quad \text{and} \quad \int_0^{\infty} x^4 (f'(x))^2 \rd x < \infty,
    \end{equation}
    where $T_1,T_2 \stackrel{\mathrm{i.i.d.}}{\sim} \mathrm{IGau}(x, b^{-1} x)$, then we have
    \begin{align}
        \mathrm{MISE}(\hat{F}_{n,b}^{\mathrm{IGau}})
        &= \int_0^{\infty} \VV(\hat{F}_{n,b}^{\mathrm{IGau}}(x)) \rd x + \int_0^{\infty} \big(\BB[\hat{F}_{n,b}^{\mathrm{IGau}}(x)]\big)^2 \rd x \notag \\
        &= n^{-1} \int_0^{\infty} F(x) (1 - F(x)) \rd x \notag \\[-1mm]
        &\quad- n^{-1} b^{1/2} \int_0^{\infty} \frac{f(x)}{2} \Big[\lim_{b\to 0} b^{-1/2} \, \EE[|T_1 - T_2|]\Big] \rd x \notag \\
        &\quad+ b^2 \int_0^{\infty} \frac{x^4}{4} (f'(x))^2 \rd x + \oo(n^{-1} b^{1/2}) + \oo(b^2).
    \end{align}
    The quantity $\lim_{b\to 0} b^{-1/2} \, \EE[|T_1 - T_2|]$ needs to be approximated numerically.
    In particular, if $\int_0^{\infty} x^4(f'(x))^2 \rd x > 0$, the asymptotically optimal choice of $b$, with respect to $\mathrm{MISE}$, is
    \begin{equation}\label{eq:b.opt.MISE.IGau}
        b_{\mathrm{opt}} = n^{-2/3} \left[\frac{4 \int_0^{\infty} \frac{x^4}{4} (f'(x))^2 \rd x}{\int_0^{\infty} \frac{f(x)}{2} \lim_{b\to 0} b^{-1/2} \, \EE[|T_1 - T_2|] \rd x}\right]^{-2/3}
    \end{equation}
    with
    \begin{equation}
        \begin{aligned}
            \mathrm{MISE}(\hat{F}_{n,b_{\mathrm{opt}}}^{\mathrm{IGau}})
            &= n^{-1} \int_0^{\infty} F(x) (1 - F(x)) \rd x \\[-1.5mm]
            &\quad- n^{-4/3} \, \frac{3}{4} \left[\frac{\big(\int_0^{\infty} \frac{f(x)}{2} \lim_{b\to 0} b^{-1/2} \, \EE[|T_1 - T_2|] \rd x\big)^4}{4 \int_0^{\infty} \frac{x^4}{4} (f'(x))^2 \rd x}\right]^{1/3} \\
            &\quad+ \oo(n^{-4/3}).
        \end{aligned}
    \end{equation}
\end{proposition}

\begin{proposition}[{\color{blue} Asymptotic normality}]\label{prop:asymptotic.normality.IGau.kernel}
    For any $x > 0$ such that $0 < F(x) < 1$, we have the following convergence in distribution:
    \begin{equation}
        n^{1/2} (\hat{F}_{n,b}^{\mathrm{IGau}}(x) - \EE[\hat{F}_{n,b}^{\mathrm{IGau}}(x)]) \stackrel{\mathscr{D}}{\longrightarrow} \mathcal{N}(0, \sigma^2(x)), \quad \text{as } b\to 0, \, n\to \infty,
    \end{equation}
    where $\sigma^2(x) \leqdef F(x) (1 - F(x))$.
    In particular, Lemma~\ref{lem:bias.variance.IGau.kernel} implies
    \begin{alignat}{3}
        &n^{1/2} (\hat{F}_{n,b}^{\mathrm{IGau}}(x) - F(x)) \stackrel{\mathscr{D}}{\longrightarrow} \mathcal{N}(0,\sigma^2(x)), \quad &&\text{if } n^{1/2} b \to 0, \\
        &n^{1/2} (\hat{F}_{n,b}^{\mathrm{IGau}}(x) - F(x)) \stackrel{\mathscr{D}}{\longrightarrow} \mathcal{N}(\lambda \cdot \frac{x^2}{2} f'(x), \sigma^2(x)), \quad &&\text{if } n^{1/2} b \to \lambda,
    \end{alignat}
    for any constant $\lambda > 0$.
\end{proposition}

\section{Asymptotic properties of the c.d.f.\ estimator with RIG kernel}\label{sec:RIG.kernel.results}

In this section, we find the asymptotic properties of the reciprocal inverse Gaussian (RIG) kernel estimator defined in \eqref{def:RIG.kernel.estimator}.

\begin{lemma}[{\color{blue} Bias and variance}]\label{lem:bias.variance.RIG.kernel}
    For any given $x\in (0,\infty)$,
    \begin{align}
        \BB[\hat{F}_{n,b}^{\mathrm{RIG}}(x)]
        &\leqdef \EE[\hat{F}_{n,b}^{\mathrm{IGau}}(x)] - F(x) = b \cdot \frac{x^2}{2} f'(x) + \oo_x(b), \\[1mm]
        \VV(\hat{F}_{n,b}^{\mathrm{RIG}}(x))
        &= n^{-1} F(x) (1 - F(x)) \notag \\
        &\quad- n^{-1} b^{1/2} \cdot \frac{f(x)}{2} \Big[\lim_{b\to 0} b^{-1/2} \, \EE[|T_1 - T_2|]\Big] \notag \\
        &\quad+ \OO_x(n^{-1} b),
    \end{align}
    where $T_1,T_2 \stackrel{\mathrm{i.i.d.}}{\sim} \mathrm{RIG}(x, b^{-1} x)$.
\end{lemma}

\begin{corollary}[{\color{blue} Mean squared error}]\label{prop:MSE.RIG.kernel}
    For any given $x\in (0,\infty)$,
    \begin{align}
        \mathrm{MSE}(\hat{F}_{n,b}^{\mathrm{RIG}}(x))
        &= \VV(\hat{F}_{n,b}^{\mathrm{RIG}}(x)) + \big(\BB[\hat{F}_{n,b}^{\mathrm{RIG}}(x)]\big)^2 \notag \\[2mm]
        &= n^{-1} F(x) (1 - F(x)) \notag \\
        &\quad- n^{-1} b^{1/2} \cdot \frac{f(x)}{2} \Big[\lim_{b\to 0} b^{-1/2} \, \EE[|T_1 - T_2|]\Big] \notag \\[-1mm]
        &\quad+ b^2 \cdot \frac{1}{4} x^4 (f'(x))^2 + \OO_x(n^{-1} b) + \oo_x(b^2),
    \end{align}
    where $T_1,T_2 \stackrel{\mathrm{i.i.d.}}{\sim} \mathrm{RIG}(x^{-1} (1 - b)^{-1}, x^{-1} b^{-1})$.
    The quantity $\lim_{b\to 0} b^{-1/2} \, \EE[|T_1 - T_2|]$ needs to be approximated numerically.
    In particular, if $f(x) \big[\lim_{b\to 0} b^{-1/2} \, \EE[|T_1 - T_2|]\big] \cdot f'(x) \neq 0$, the asymptotically optimal choice of $b$, with respect to $\mathrm{MSE}$, is
    \begin{equation}
        b_{\mathrm{opt}} = n^{-2/3} \left[\frac{4 \cdot \frac{1}{4} x^4 (f'(x))^2}{\frac{f(x)}{2} \lim_{b\to 0} b^{-1/2} \, \EE[|T_1 - T_2|]}\right]^{-2/3}
    \end{equation}
    with
    \begin{equation}
        \begin{aligned}
            \mathrm{MSE}(\hat{F}_{n,b_{\mathrm{opt}}}^{\mathrm{RIG}}(x))
            &= n^{-1} F(x) (1 - F(x)) \\[-1.5mm]
            &\quad- n^{-4/3} \, \frac{3}{4} \left[\frac{\big(\frac{f(x)}{2} \lim_{b\to 0} b^{-1/2} \, \EE[|T_1 - T_2|]\big)^4}{4 \cdot \frac{1}{4} x^4 (f'(x))^2}\right]^{1/3} \\
            &\quad+ \oo_x(n^{-4/3}).
        \end{aligned}
    \end{equation}
\end{corollary}

\begin{proposition}[{\color{blue} Mean integrated squared error}]\label{prop:MISE.RIG.kernel}
    Assuming that the target density $f = F'$ satisfies
    \begin{equation}
        \int_0^{\infty} f(x) \Big[\lim_{b\to 0} b^{-1/2} \, \EE[|T_1 - T_2|]\Big] \rd x < \infty \quad \text{and} \quad \int_0^{\infty} x^4 (f'(x))^2 \rd x < \infty,
    \end{equation}
    where $T_1,T_2 \stackrel{\mathrm{i.i.d.}}{\sim} \mathrm{RIG}(x^{-1} (1 - b)^{-1}, x^{-1} b^{-1})$, then we have
    \begin{align}
        \mathrm{MISE}(\hat{F}_{n,b}^{\mathrm{RIG}}(x))
        &= \int_0^{\infty} \VV(\hat{F}_{n,b}^{\mathrm{RIG}}(x)) \rd x + \int_0^{\infty} \big(\BB[\hat{F}_{n,b}^{\mathrm{RIG}}(x)]\big)^2 \rd x \notag \\
        &= n^{-1} \int_0^{\infty} F(x) (1 - F(x)) \rd x \notag \\[-1mm]
        &\quad- n^{-1} b^{1/2} \int_0^{\infty} \frac{f(x)}{2} \Big[\lim_{b\to 0} b^{-1/2} \, \EE[|T_1 - T_2|]\Big] \rd x \notag \\
        &\quad+ b^2 \int_0^{\infty} \frac{x^4}{4} (f'(x))^2 \rd x + \oo(n^{-1} b^{1/2}) + \oo(b^2).
    \end{align}
    The quantity $\lim_{b\to 0} b^{-1/2} \, \EE[|T_1 - T_2|]$ needs to be approximated numerically.
    In particular, if $\int_0^{\infty} x^4(f'(x))^2 \rd x > 0$, the asymptotically optimal choice of $b$, with respect to $\mathrm{MISE}$, is
    \begin{equation}\label{eq:b.opt.MISE.RIG}
        b_{\mathrm{opt}} = n^{-2/3} \left[\frac{4 \int_0^{\infty} \frac{1}{4} x^4 (f'(x))^2 \rd x}{\int_0^{\infty} \frac{f(x)}{2} \lim_{b\to 0} b^{-1/2} \, \EE[|T_1 - T_2|] \rd x}\right]^{-2/3}
    \end{equation}
    with
    \begin{equation}
        \begin{aligned}
            \mathrm{MISE}(\hat{F}_{n,b_{\mathrm{opt}}}^{\mathrm{RIG}}(x))
            &= n^{-1} \int_0^{\infty} F(x) (1 - F(x)) \rd x \\[-1.5mm]
            &\quad- n^{-4/3} \, \frac{3}{4} \left[\frac{\big(\int_0^{\infty} \frac{f(x)}{2} \lim_{b\to 0} b^{-1/2} \, \EE[|T_1 - T_2|] \rd x\big)^4}{4 \int_0^{\infty} \frac{1}{4} x^4 (f'(x))^2 \rd x}\right]^{1/3} \\
            &\quad+ \oo(n^{-4/3}).
        \end{aligned}
    \end{equation}
\end{proposition}

\begin{proposition}[{\color{blue} Asymptotic normality}]\label{prop:asymptotic.normality.RIG.kernel}
    For any $x > 0$ such that $0 < F(x) < 1$, we have the following convergence in distribution:
    \begin{equation}
        n^{1/2} (\hat{F}_{n,b}^{\mathrm{RIG}}(x) - \EE[\hat{F}_{n,b}^{\mathrm{IGau}}(x)]) \stackrel{\mathscr{D}}{\longrightarrow} \mathcal{N}(0, \sigma^2(x)), \quad \text{as } b\to 0, \, n\to \infty,
    \end{equation}
    where $\sigma^2(x) \leqdef F(x) (1 - F(x))$.
    In particular, Lemma~\ref{lem:bias.variance.RIG.kernel} implies
    \begin{alignat}{3}
        &n^{1/2} (\hat{F}_{n,b}^{\mathrm{RIG}}(x) - F(x)) \stackrel{\mathscr{D}}{\longrightarrow} \mathcal{N}(0,\sigma^2(x)), \quad &&\text{if } n^{1/2} b \to 0, \\
        &n^{1/2} (\hat{F}_{n,b}^{\mathrm{RIG}}(x) - F(x)) \stackrel{\mathscr{D}}{\longrightarrow} \mathcal{N}(\lambda \cdot \frac{x^2}{2} f'(x), \sigma^2(x)), \quad &&\text{if } n^{1/2} b \to \lambda,
    \end{alignat}
    for any constant $\lambda > 0$.
\end{proposition}

\section{Numerical study}\label{sec:numerical.study}

As in \cite{Mombeni_et_al_2019_accepted}, we generated $M = 1000$ samples of size $n = 256$ and $n = 1000$ from eight distributions:
\begin{enumerate}
    \item $\text{Burr}\hspace{0.2mm}(1,3,1)$, with the following parametrization for the density function:
        \begin{equation}
            \begin{array}{l}
                f_1(x \nvert \lambda, c, k) \leqdef \frac{c k}{\lambda} \big(\frac{x}{\lambda}\big)^{c-1} \big[1 + \big(\frac{x}{\lambda}\big)^c\big]^{-k - 1} \ind_{(0,\infty)}(x), \quad \lambda, c, k > 0;
            \end{array}
        \end{equation}
    \item $\text{Gamma}\hspace{0.2mm}(0.6,2)$, with the following parametrization for the density function:
        \begin{equation}
            \begin{array}{l}
                f_2(x \nvert \alpha, \theta) \leqdef \frac{x^{\alpha - 1} \exp(-\frac{x}{\theta})}{\theta^{\alpha} \Gamma(\alpha)} \ind_{(0,\infty)}(x), \quad \alpha,\theta > 0;
            \end{array}
        \end{equation}
    \item $\text{Gamma}\hspace{0.2mm}(4,2)$, with the following parametrization for the density function:
        \begin{equation}
            \begin{array}{l}
                f_3(x \nvert \alpha, \theta) \leqdef \frac{x^{\alpha - 1} \exp(-\frac{x}{\theta})}{\theta^{\alpha} \Gamma(\alpha)} \ind_{(0,\infty)}(x), \quad \alpha,\theta > 0;
            \end{array}
        \end{equation}
    \item $\text{GeneralizedPareto}\hspace{0.2mm}(0.4, 1, 0)$, with the following parametrization for the density function:
        \begin{equation}
            \begin{array}{l}
                f_4(x \nvert \xi, \sigma, \mu) \leqdef  \frac{1}{\sigma} \big[1 + \xi \big(\frac{x - \mu}{\sigma}\big)\big]^{-\frac{1}{\xi} - 1} \ind_{(\mu,\infty)}(x), \quad \xi, \sigma > 0, ~\mu\in \R;
            \end{array}
        \end{equation}
    \item $\text{HalfNormal}\hspace{0.2mm}(1)$, with the following parametrization for the density function:
        \begin{equation}
            \begin{array}{l}
                f_5(x \nvert \sigma) \leqdef \sqrt{\frac{2}{\pi \sigma^2}} \exp\big(-\frac{x^2}{2 \sigma^2}\big) \ind_{(0,\infty)}(x), \quad \sigma > 0;
            \end{array}
        \end{equation}
    \item $\text{LogNormal}\hspace{0.2mm}(0, 0.75)$, with the following parametrization for the density function:
        \begin{equation}
            \begin{array}{l}
                f_6(x \nvert \mu, \sigma) \leqdef \frac{1}{x \sqrt{2 \pi \sigma^2}} \exp\big(-\frac{(\log x - \mu)^2}{2 \sigma^2}\big) \ind_{(0,\infty)}(x), \quad \mu\in \R, ~\sigma > 0;
            \end{array}
        \end{equation}
    \item $\text{Weibull}\hspace{0.2mm}(1.5, 1.5)$, with the following parametrization for the density function:
        \begin{equation}
            \begin{array}{l}
                f_7(x \nvert \lambda, k) \leqdef \frac{k}{\lambda} \, \big(\frac{x}{\lambda}\big)^{k-1} \exp\big(-\big(\frac{x}{\lambda}\big)^k\big) \ind_{(0,\infty)}(x), \quad \lambda, k > 0;
            \end{array}
        \end{equation}
    \item $\text{Weibull}\hspace{0.2mm}(3, 2)$, with the following parametrization for the density function:
        \begin{equation}
            \begin{array}{l}
                f_8(x \nvert \lambda, k) \leqdef \frac{k}{\lambda} \, \big(\frac{x}{\lambda}\big)^{k-1} \exp\big(-\big(\frac{x}{\lambda}\big)^k\big) \ind_{(0,\infty)}(x), \quad \lambda, k > 0.
            \end{array}
        \end{equation}
\end{enumerate}
For each of the eight distributions ($i = 1, 2, \dots, 8$), each of the ten estimators ($j = 1, 2, \dots, 10$), each sample size ($n = 256, 1000$), and each sample ($k = 1, 2, \dots, M$), we calculated the integrated squared errors
\begin{equation}
    \mathrm{ISE}_{i,j,n}^{(k)} \leqdef \int_0^{\infty} (\hat{F}_{j,n}^{(k)}(x) - F_i(x))^2 \rd x,
\end{equation}
where
\begin{enumerate}
    \item $\hat{F}_{1,n}^{(k)}$ denotes the estimator $\hat{F}_{n,b_{\mathrm{opt}}}^{\mathrm{Gam}}$ from \eqref{def:G.kernel.estimator} applied to the $k$-th sample;
    \item $\hat{F}_{2,n}^{(k)}$ denotes the estimator $\hat{F}_{n,b_{\mathrm{opt}}}^{\mathrm{IGam}}$ from \eqref{def:IGam.kernel.estimator} applied to the $k$-th sample;
    \item $\hat{F}_{3,n}^{(k)}$ denotes the estimator $\hat{F}_{n,b_{\mathrm{opt}}}^{\mathrm{LN}}$ from \eqref{def:LN.kernel.estimator} applied to the $k$-th sample;
    \item $\hat{F}_{4,n}^{(k)}$ denotes the estimator $\hat{F}_{n,b_{\mathrm{opt}}}^{\mathrm{IGau}}$ from \eqref{def:IGau.kernel.estimator} applied to the $k$-th sample;
    \item $\hat{F}_{5,n}^{(k)}$ denotes the estimator $\hat{F}_{n,b_{\mathrm{opt}}}^{\mathrm{RIG}}$ from \eqref{def:RIG.kernel.estimator} applied to the $k$-th sample;
    \item $\hat{F}_{6,n}^{(k)}$ denotes the estimator $\hat{F}_{n,b_{\mathrm{opt}}}^{\mathrm{B-S}}$ from \eqref{def:BS.kernel.estimator} applied to the $k$-th sample;
    \item $\hat{F}_{7,n}^{(k)}$ denotes the estimator $\hat{F}_{n,b_{\mathrm{opt}}}^{\mathrm{W}}$ from \eqref{def:W.kernel.estimator} applied to the $k$-th sample;
\end{enumerate}
(for $\hat{F}_{1,n}^{(\scriptscriptstyle k)}$ to $\hat{F}_{7,n}^{(\scriptscriptstyle k)}$, $b_{\mathrm{opt}}$ is optimal with respect to the MISE (see \eqref{eq:b.opt.MISE.G}, \eqref{eq:b.opt.MISE.IGam}, \eqref{eq:b.opt.MISE.LN}, \eqref{eq:b.opt.MISE.IGau} and \eqref{eq:b.opt.MISE.RIG}) and approximated under the assumption that the target distribution is $\text{Gamma}\hspace{0.2mm}(\hat{\alpha}_n^{(\scriptscriptstyle k)},\hat{\theta}_n^{(\scriptscriptstyle k)})$, where $\hat{\alpha}_n^{(\scriptscriptstyle k)}$ and $\hat{\theta}_n^{(\scriptscriptstyle k)}$ are the maximum likelihood estimates for the $k$-th sample)
and
\begin{enumerate}\setcounter{enumi}{7}
    \item $\hat{F}_{8,n}^{(k)}(x) \leqdef \frac{1}{n} \sum_{\ell=1}^n \mathrm{Epa}\big(\frac{x - X_{\ell}^{(\scriptscriptstyle k)}}{b_{\mathrm{LNO}}}\big)$, where
        \begin{itemize}
            \item $\mathrm{Epa}(u) \leqdef \big(\frac{1}{2} + \frac{3u}{4} - \frac{u^3}{4}\big) \cdot \ind_{(-1,1)}(u) + \ind_{[1,\infty)}(u)$ denotes the c.d.f.\ of the Epanechnikov kernel,
            \item $b_{\mathrm{LNO}}$ is selected by minimizing the Leave-None-Out criterion from page 197 in \cite{MR1354087};
        \end{itemize}
    \item $\hat{F}_{9,n}^{(k)}(x) \leqdef \frac{1}{n} \sum_{\ell=1}^n \Big\{\mathrm{Epa}\big(\frac{x - X_{\ell}^{(\scriptscriptstyle k)}}{b_{\mathrm{CV}}}\big) \ind_{[b_{\mathrm{CV}},\infty)}(x) + \mathrm{Epa}\big(\frac{x - X_{\ell}^{(\scriptscriptstyle k)}}{x}\big) \ind_{(0,b_{\mathrm{CV}})}(x)\Big\}$ is the boundary modified kernel estimator from Example 2.3 in \cite{MR3072469}, where
        \begin{itemize}
            \item $\mathrm{Epa}(u) \leqdef \big(\frac{1}{2} + \frac{3u}{4} - \frac{u^3}{4}\big) \cdot \ind_{(-1,1)}(u) + \ind_{[1,\infty)}(u)$ denotes the c.d.f.\ of the Epanechnikov kernel,
            \item $b_{\mathrm{CV}}$ is selected by minimizing the Cross-Validation criterion from page 180 in \cite{MR3072469};
        \end{itemize}
    \item $\hat{F}_{10,n}^{(k)}(x) \leqdef \frac{1}{n} \sum_{\ell=1}^n \ind_{\{X_{\ell}^{(\scriptscriptstyle k)} \leq x\}}$ is the empirical c.d.f.\ applied to the $k$-th sample;
\end{enumerate}
Everywhere in our \texttt{R} code, we approximated the integrals on $(0,\infty)$ using the \texttt{integral} function from the \texttt{R} package \texttt{pracma} (the base function \texttt{integrate} had serious precision issues).
Table~\ref{table:1} below shows the mean and standard deviation of the $\mathrm{ISE}$'s, i.e.,
\begin{equation}
    \frac{1}{M} \sum_{k=1}^M \mathrm{ISE}_{i,j,n}^{(k)} \quad \text{and} \quad \sqrt{\frac{1}{M-1} \sum_{k=1}^M \Big(\mathrm{ISE}_{i,j,n}^{(k)} - \frac{1}{M} \sum_{k'=1}^M \mathrm{ISE}_{i,j,n}^{(k')}\Big)^2},
\end{equation}
for the eight distributions ($i = 1, 2, \dots, 8$), the ten estimators ($j = 1, 2, \dots, 10$) and the two sample sizes ($n = 256, 1000$).
All the values presented in the table have been multiplied by $10^4$.
In Table~\ref{table:2}, we computed, for each distribution and each sample size, the difference between the $\mathrm{ISE}$ means and the lowest $\mathrm{ISE}$ mean for the corresponding distribution and sample size (i.e., the $\mathrm{ISE}$ means minus the $\mathrm{ISE}$ mean of the best estimator on the corresponding line).
The totals of those differences are also calculated for each sample size on the two ``total'' lines.
Figure~\ref{fig:ISE.boxplots} gives a better idea of the distribution of $\mathrm{ISE}$'s by displaying the boxplot of the $\mathrm{ISE}$'s for every distribution and every estimator, when the sample size is $n = 1000$. Finally, the Figures~\ref{fig:cdf.estimates.burr},~\ref{fig:cdf.estimates.gamma.1},~\ref{fig:cdf.estimates.gamma.2},~\ref{fig:cdf.estimates.gen.pareto},~\ref{fig:cdf.estimates.halfnormal},~\ref{fig:cdf.estimates.lognormal},~\ref{fig:cdf.estimates.weibull.1},~\ref{fig:cdf.estimates.weibull.2} (one figure for each of the eight distributions) show a collection of ten c.d.f.\ estimates from each of the ten estimators.

\vspace{2mm}
Here are the results, which we discuss in Section~\ref{sec:discussion}:

    {
    \newgeometry{bottom=3.5cm}
    \pagestyle{empty}
    \begin{landscape}
        \def\arraystretch{1.45}
        \setlength\tabcolsep{1mm}
        \begin{table}
            \captionsetup{width=0.9\linewidth}
            \centering
            \begin{tabular}{|cc|cc|cc|cc|cc|cc|cc|cc|cc|cc|cc|}
                \hline
                \multicolumn{ 1}{|c}{$\times 10^{-4}$} & \multicolumn{ 1}{|c}{$i$} & \multicolumn{ 2}{|c}{Gam\hspace{0.3mm}($i \hspace{-0.5mm}=\hspace{-0.5mm} 1$)} & \multicolumn{ 2}{|c}{IGam\hspace{0.3mm}($i \hspace{-0.5mm}=\hspace{-0.5mm} 2$)} & \multicolumn{ 2}{|c}{LN\hspace{0.3mm}($i \hspace{-0.5mm}=\hspace{-0.5mm} 3$)} & \multicolumn{ 2}{|c}{IGau\hspace{0.3mm}($i \hspace{-0.5mm}=\hspace{-0.5mm} 4$)} & \multicolumn{ 2}{|c}{RIG\hspace{0.3mm}($i \hspace{-0.5mm}=\hspace{-0.5mm} 5$)} & \multicolumn{ 2}{|c}{B-S\hspace{0.3mm}($i \hspace{-0.5mm}=\hspace{-0.5mm} 6$)} & \multicolumn{ 2}{|c}{W\hspace{0.3mm}($i \hspace{-0.5mm}=\hspace{-0.5mm} 7$)} & \multicolumn{ 2}{|c}{OK\hspace{0.3mm}($i \hspace{-0.5mm}=\hspace{-0.5mm} 8$)} & \multicolumn{ 2}{|c}{BK\hspace{0.3mm}($i \hspace{-0.5mm}=\hspace{-0.5mm} 9$)} & \multicolumn{ 2}{|c|}{EDF\hspace{0.3mm}($i \hspace{-0.5mm}=\hspace{-0.5mm} 10$)} \\
                \hline
                \multicolumn{ 1}{|c}{$n$} & \multicolumn{ 1}{|c|}{$j$} &      mean &       std. &    mean &       std. &  mean &       std. &       mean &       std. &       mean &       std. &       mean &       std. &       mean &       std. &       mean &       std. &       mean &       std. &       mean &       std. \\
                \hline
                   \multicolumn{ 1}{|c|}{256} &          1 &       1.39 &       1.27 &       1.37 &       1.34 &     \cellcolor{cyan!25}  1.31 &       1.26 &       1.37 &       1.32 &       1.37 &       1.32 &     \cellcolor{cyan!25}  1.31 &       1.26 &       1.37 &       1.32 &       1.54 &       1.43 &       1.47 &       1.34 &       1.54 &       1.44 \\

                       \multicolumn{ 1}{|c|}{} &          2 &       2.59 &       2.36 &       2.50 &       2.53 &     \cellcolor{cyan!25}  2.36 &       2.42 &       2.49 &       2.46 &       2.49 &       2.47 &     \cellcolor{cyan!25}  2.36 &       2.42 &       2.50 &       2.51 &       2.76 &       2.44 &       2.67 &       2.57 &       2.76 &       2.45 \\

                       \multicolumn{ 1}{|c|}{} &          3 &       6.70 &       6.28 &       6.77 &       6.58 &      \cellcolor{cyan!25} 6.62 &       6.28 &       6.69 &       6.45 &       6.69 &       6.45 &    \cellcolor{cyan!25}   6.62 &       6.28 &       6.74 &       6.54 &       7.44 &       7.01 &       6.70 &       6.39 &       7.44 &       7.00 \\

                       \multicolumn{ 1}{|c|}{} &          4 &       3.74 &       3.14 &       3.60 &       3.27 &     \cellcolor{cyan!25}  3.36 &       3.15 &       3.61 &       3.20 &       3.61 &       3.21 &     \cellcolor{cyan!25}  3.36 &       3.14 &       3.60 &       3.26 &       3.96 &       3.24 &       3.80 &       3.27 &       3.97 &       3.24 \\

                       \multicolumn{ 1}{|c|}{} &          5 &       1.14 &       1.10 &       1.18 &       1.13 &      1.18 &       1.07 &       1.17 &       1.13 &       1.17 &       1.13 &      1.18 &       1.07 &       1.17 &       1.12 &       1.26 &       1.19 &    \cellcolor{cyan!25}   1.10 &       1.14 &       1.26 &       1.19 \\

                       \multicolumn{ 1}{|c|}{} &          6 &       1.93 &       1.83 &       1.91 &       1.89 &      \cellcolor{cyan!25} 1.81 &       1.80 &       1.91 &       1.87 &       1.91 &       1.87 &     \cellcolor{cyan!25}  1.81 &       1.80 &       1.91 &       1.88 &       2.13 &       1.94 &       2.05 &       1.93 &       2.13 &       1.95 \\

                       \multicolumn{ 1}{|c|}{} &          7 &       1.75 &       1.82 &       1.77 &       1.99 &     \cellcolor{cyan!25}  1.68 &       1.83 &       1.76 &       1.96 &       1.76 &       1.96 &     \cellcolor{cyan!25}  1.68 &       1.83 &       1.76 &       1.95 &       1.95 &       1.93 &       1.73 &       2.04 &       1.95 &       1.92 \\

                       \multicolumn{ 1}{|c|}{} &          8 &       2.69 &       2.71 &       2.75 &       2.78 &     2.81 &       2.66 &       2.67 &       2.71 &       2.67 &       2.71 &       2.81 &       2.66 &       2.75 &       2.75 &       3.02 &       2.88 &    \cellcolor{cyan!25}   2.56 &       2.59 &       3.03 &       2.88 \\
            \hline
                  \multicolumn{ 1}{|c|}{1000} &          1 &       0.40 &       0.36 &       0.39 &       0.36 &       \cellcolor{cyan!25}  0.38 &       0.35 &       0.39 &       0.36 &       0.39 &       0.36 &     \cellcolor{cyan!25}  0.38 &       0.35 &       0.39 &       0.36 &       0.43 &       0.39 &       0.41 &       0.36 &       0.43 &       0.39 \\

                       \multicolumn{ 1}{|c|}{} &          2 &       0.72 &       0.70 &       0.70 &       0.69 &      \cellcolor{cyan!25} 0.67 &       0.67 &       0.70 &       0.69 &       0.70 &       0.69 &     \cellcolor{cyan!25}  0.67 &       0.67 &       0.70 &       0.69 &       0.75 &       0.71 &       0.73 &       0.72 &       0.75 &       0.71 \\

                       \multicolumn{ 1}{|c|}{} &          3 &       2.01 &       2.09 &       2.05 &       2.22 &      2.02 &       2.16 &       2.04 &       2.15 &       2.04 &       2.15 &      2.02 &       2.16 &       2.05 &       2.20 &       2.23 &       2.29 &   \cellcolor{cyan!25}    1.99 &       2.09 &       2.23 &       2.30 \\

                       \multicolumn{ 1}{|c|}{} &          4 &       0.99 &       0.79 &       0.97 &       0.82 &     \cellcolor{cyan!25}  0.93 &       0.80 &       0.97 &       0.81 &       0.97 &       0.81 &      \cellcolor{cyan!25} 0.93 &       0.80 &       0.97 &       0.82 &       1.03 &       0.82 &       1.00 &       0.82 &       1.03 &       0.83 \\

                       \multicolumn{ 1}{|c|}{} &          5 &       0.31 &       0.31 &       0.31 &       0.31 &      0.31 &       0.30 &       0.31 &       0.31 &       0.31 &       0.31 &       0.31 &       0.30 &       0.31 &       0.31 &       0.33 &       0.32 &     \cellcolor{cyan!25}  0.30 &       0.32 &       0.33 &       0.32 \\

                       \multicolumn{ 1}{|c|}{} &          6 &       0.47 &       0.43 &       0.47 &       0.43 &     \cellcolor{cyan!25}  0.46 &       0.42 &       0.47 &       0.43 &       0.47 &       0.43 &     \cellcolor{cyan!25}  0.46 &       0.42 &       0.47 &       0.43 &       0.50 &       0.45 &       0.49 &       0.43 &       0.50 &       0.45 \\

                       \multicolumn{ 1}{|c|}{} &          7 &       0.46 &       0.46 &       0.46 &       0.48 &     \cellcolor{cyan!25}  0.44 &       0.45 &       0.46 &       0.48 &       0.46 &       0.48 &     \cellcolor{cyan!25}  0.44 &       0.45 &       0.46 &       0.48 &       0.49 &       0.50 &       0.46 &       0.48 &       0.49 &       0.50 \\

                       \multicolumn{ 1}{|c|}{} &          8 &       0.72 &       0.74 &       0.74 &       0.75 &     0.75 &       0.74 &       0.73 &       0.75 &       0.73 &       0.75 &      0.75 &       0.74 &       0.74 &       0.75 &       0.78 &       0.80 &     \cellcolor{cyan!25}  0.70 &       0.72 &       0.78 &       0.81 \\
            \hline
            \end{tabular}
            \caption{The mean and standard deviation of the $\mathrm{ISE}_{i,j,n}^{(k)}, ~k = 1,2,\dots,M$, for the eight distributions ($i = 1, 2, \dots, 8$), the ten estimators ($j = 1, 2, \dots, 10$) and the two sample sizes ($n = 256, 1000$). All the values presented in the table have been multiplied by $10^4$. The ordinary kernel estimator $\hat{F}_8$ is denoted by $\mathrm{OK}$, the boundary kernel estimator $\hat{F}_9$ is denoted by $\mathrm{BK}$, and the empirical c.d.f.\ $\hat{F}_{10}$ is denoted by $\mathrm{EDF}$. For each line in the table, the lowest $\mathrm{ISE}$ means are highlighted in blue.}
            \label{table:1}
        \end{table}
    \end{landscape}
    }

    {
    \newgeometry{bottom=3.5cm}
    \pagestyle{empty}
    \begin{landscape}
        \def\arraystretch{1.22}
        \setlength\tabcolsep{1.1mm}
        \begin{table}
            \captionsetup{width=0.9\linewidth}
            \centering
            \begin{tabular}{|cc|c|c|c|c|c|c|c|c|c|c|}
                \hline
                \multicolumn{ 1}{|c}{$\times 10^{-4}$} & \multicolumn{ 1}{|c}{$i$} & \multicolumn{ 1}{|c}{Gam\hspace{0.3mm}($i \hspace{-0.5mm}=\hspace{-0.5mm} 1$)} & \multicolumn{ 1}{|c}{IGam\hspace{0.3mm}($i \hspace{-0.5mm}=\hspace{-0.5mm} 2$)} & \multicolumn{ 1}{|c}{LN\hspace{0.3mm}($i \hspace{-0.5mm}=\hspace{-0.5mm} 3$)} & \multicolumn{ 1}{|c}{IGau\hspace{0.3mm}($i \hspace{-0.5mm}=\hspace{-0.5mm} 4$)} & \multicolumn{ 1}{|c}{RIG\hspace{0.3mm}($i \hspace{-0.5mm}=\hspace{-0.5mm} 5$)} & \multicolumn{ 1}{|c}{B-S\hspace{0.3mm}($i \hspace{-0.5mm}=\hspace{-0.5mm} 6$)} & \multicolumn{ 1}{|c}{W\hspace{0.3mm}($i \hspace{-0.5mm}=\hspace{-0.5mm} 7$)} & \multicolumn{ 1}{|c}{OK\hspace{0.3mm}($i \hspace{-0.5mm}=\hspace{-0.5mm} 8$)} & \multicolumn{ 1}{|c}{BK\hspace{0.3mm}($i \hspace{-0.5mm}=\hspace{-0.5mm} 9$)} & \multicolumn{ 1}{|c|}{EDF\hspace{0.3mm}($i \hspace{-0.5mm}=\hspace{-0.5mm} 10$)} \\
                \hline
                \multicolumn{ 1}{|c}{$n$} & \multicolumn{ 1}{|c|}{$j$} &      diff.\ with &        diff.\ with &       diff.\ with &         diff.\ with &          diff.\ with &        diff.\ with &       diff.\ with &         diff.\ with &           diff.\ with &          diff.\ with   \\[-2.5mm]
                  \multicolumn{ 1}{|c|}{} &  &      lowest &        lowest &       lowest &         lowest &          lowest &        lowest &       lowest &         lowest &           lowest &          lowest   \\[-2.5mm]
                  \multicolumn{ 1}{|c|}{} &  &      mean &        mean &       mean &         mean &          mean &        mean &       mean &         mean &           mean &          mean   \\
                \hline
                   \multicolumn{ 1}{|c|}{256} &          1 &       0.08 &       0.06 &       0.00 &       0.06 &       0.06 &       0.00 &       0.06 &       0.23 &       0.16 &       0.23 \\

                      \multicolumn{ 1}{|c|}{} &          2 &       0.23 &       0.14 &       0.00 &       0.14 &       0.13 &       0.00 &       0.14 &       0.40 &       0.32 &       0.40 \\

                      \multicolumn{ 1}{|c|}{} &          3 &       0.08 &       0.15 &       0.01 &       0.08 &       0.07 &       0.00 &       0.12 &       0.82 &       0.09 &       0.82 \\

                      \multicolumn{ 1}{|c|}{} &          4 &       0.38 &       0.24 &       0.00 &       0.25 &       0.24 &       0.00 &       0.23 &       0.60 &       0.43 &       0.60 \\

                      \multicolumn{ 1}{|c|}{} &          5 &       0.05 &       0.08 &       0.09 &       0.07 &       0.07 &       0.08 &       0.08 &       0.16 &       0.00 &       0.17 \\

                      \multicolumn{ 1}{|c|}{} &          6 &       0.12 &       0.10 &       0.00 &       0.10 &       0.10 &       0.00 &       0.10 &       0.32 &       0.24 &       0.32 \\

                      \multicolumn{ 1}{|c|}{} &          7 &       0.07 &       0.10 &       0.00 &       0.09 &       0.09 &       0.00 &       0.08 &       0.27 &       0.05 &       0.27 \\

                      \multicolumn{ 1}{|c|}{} &          8 &       0.13 &       0.19 &       0.25 &       0.11 &       0.11 &       0.25 &       0.19 &       0.46 &       0.00 &       0.46 \\
                \hline
                \multicolumn{ 2}{|c|}{total} &      1.14 &       1.07 &     \cellcolor{cyan!25}   0.35 &       0.89 &       0.88 &      \cellcolor{cyan!25}  0.34 &       0.99 &       3.26 &       1.29 &       3.28 \\

                \hline
                \hline
                  \multicolumn{ 1}{|c|}{1000} &          1 &       0.02 &       0.01 &       0.00 &       0.01 &       0.01 &       0.00 &       0.01 &       0.05 &       0.03 &       0.05 \\

                      \multicolumn{ 1}{|c|}{} &          2 &       0.04 &       0.02 &       0.00 &       0.02 &       0.02 &       0.00 &       0.02 &       0.07 &       0.06 &       0.08 \\

                      \multicolumn{ 1}{|c|}{} &          3 &       0.02 &       0.06 &       0.03 &       0.05 &       0.05 &       0.03 &       0.06 &       0.24 &       0.00 &       0.24 \\

                      \multicolumn{ 1}{|c|}{} &          4 &       0.06 &       0.04 &       0.00 &       0.04 &       0.04 &       0.00 &       0.04 &       0.10 &       0.07 &       0.10 \\

                      \multicolumn{ 1}{|c|}{} &          5 &       0.01 &       0.02 &       0.02 &       0.01 &       0.01 &       0.02 &       0.02 &       0.03 &       0.00 &       0.03 \\

                      \multicolumn{ 1}{|c|}{} &          6 &       0.02 &       0.02 &       0.00 &       0.02 &       0.02 &       0.00 &       0.02 &       0.04 &       0.04 &       0.04 \\

                      \multicolumn{ 1}{|c|}{} &          7 &       0.01 &       0.02 &       0.00 &       0.02 &       0.02 &       0.00 &       0.02 &       0.05 &       0.02 &       0.05 \\

                      \multicolumn{ 1}{|c|}{} &          8 &       0.02 &       0.04 &       0.05 &       0.03 &       0.03 &       0.05 &       0.04 &       0.08 &       0.00 &       0.08 \\
                \hline
                \multicolumn{ 2}{|c|}{total} &  0.20 &       0.23 &      \cellcolor{cyan!25}  0.10 &       0.20 &       0.20 &      \cellcolor{cyan!25}  0.10 &       0.24 &       0.66 &       0.22 &       0.68 \\

                \hline
            \end{tabular}
            \caption{For each of the eight distributions ($i = 1, 2, \dots, 8$) and each of the two sample sizes ($n = 256, 1000$), a cell represents the mean of the $\mathrm{ISE}_{i,j,n}^{(k)}, ~k = 1,2,\dots,M$, minus the lowest $\mathrm{ISE}$ mean for that line (i.e., minus the $\mathrm{ISE}$ mean of the best estimator for that specific distribution and sample size). For each estimator ($j = 1, 2, \dots, 10$) and each sample size, the total of those differences to the best $\mathrm{ISE}$ mean is calculated on the line called ``total''. For each sample size, the lowest totals are highlighted in blue.}
            \label{table:2}
        \end{table}
    \end{landscape}
    }

\newgeometry{bottom=4cm}

    \newpage
    \begin{figure}[ht]
        \captionsetup{width=0.8\linewidth}
        \vspace{-0.5cm}
        \centering
        \begin{subfigure}[b]{0.40\textwidth}
            \centering
            \includegraphics[width=\textwidth, height=0.85\textwidth]{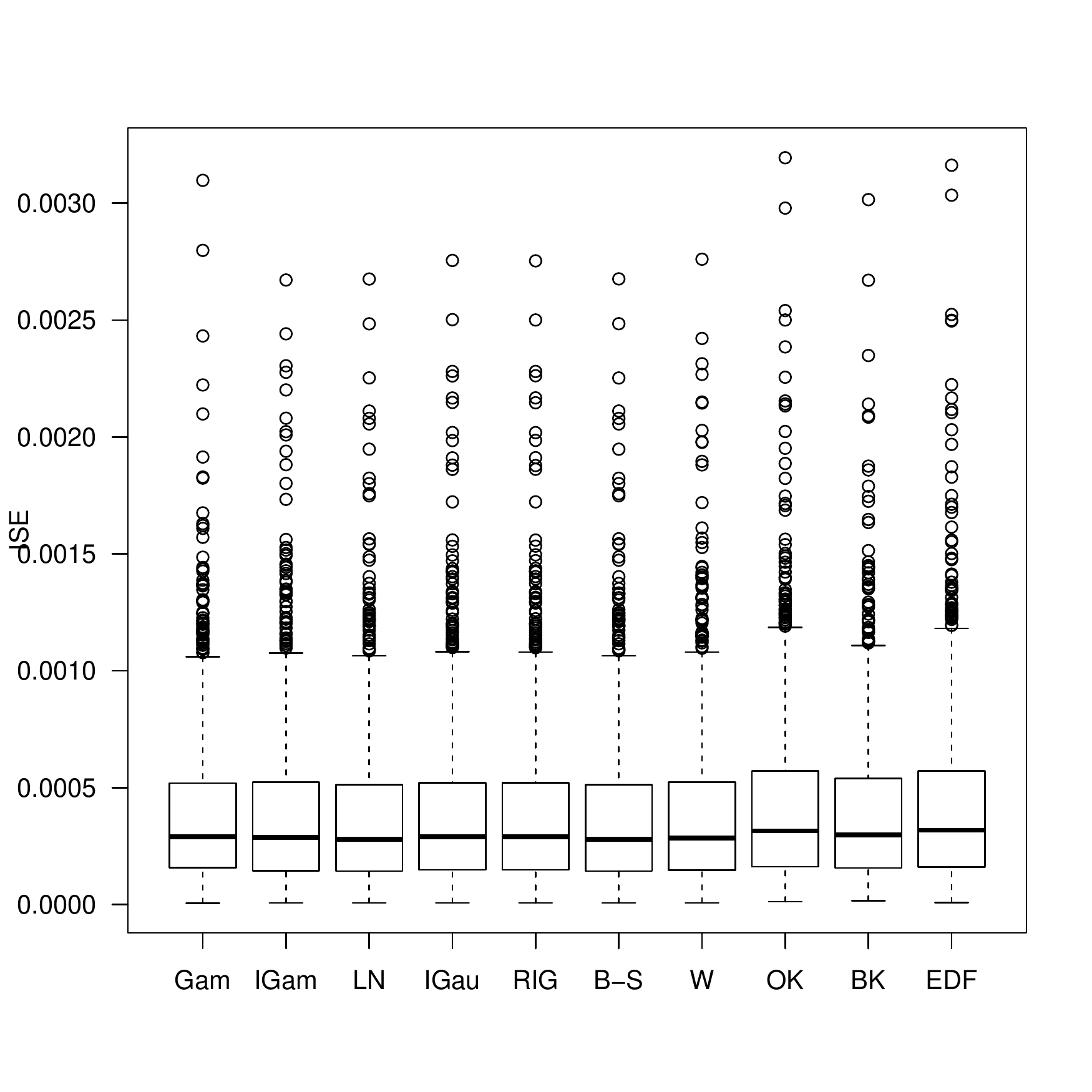}
            \vspace{-0.9cm}
            \caption{$\text{Burr}\hspace{0.2mm}(1, 3, 1)$}
        \end{subfigure}
        \quad
        \begin{subfigure}[b]{0.40\textwidth}
            \centering
            \includegraphics[width=\textwidth, height=0.85\textwidth]{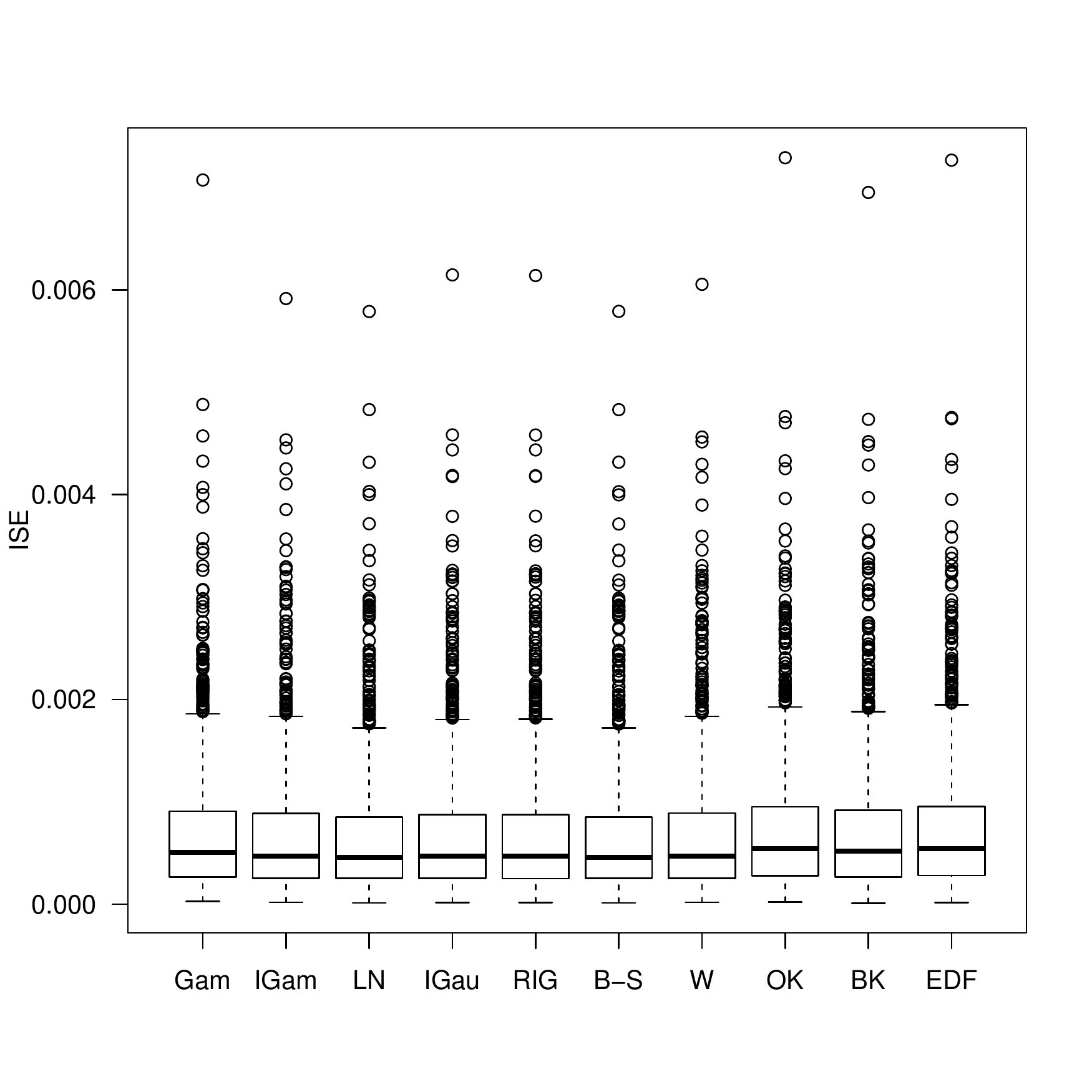}
            \vspace{-0.9cm}
            \caption{$\text{Gamma}\hspace{0.2mm}(0.6, 2)$}
        \end{subfigure}
        \begin{subfigure}[b]{0.40\textwidth}
            \centering
            \includegraphics[width=\textwidth, height=0.85\textwidth]{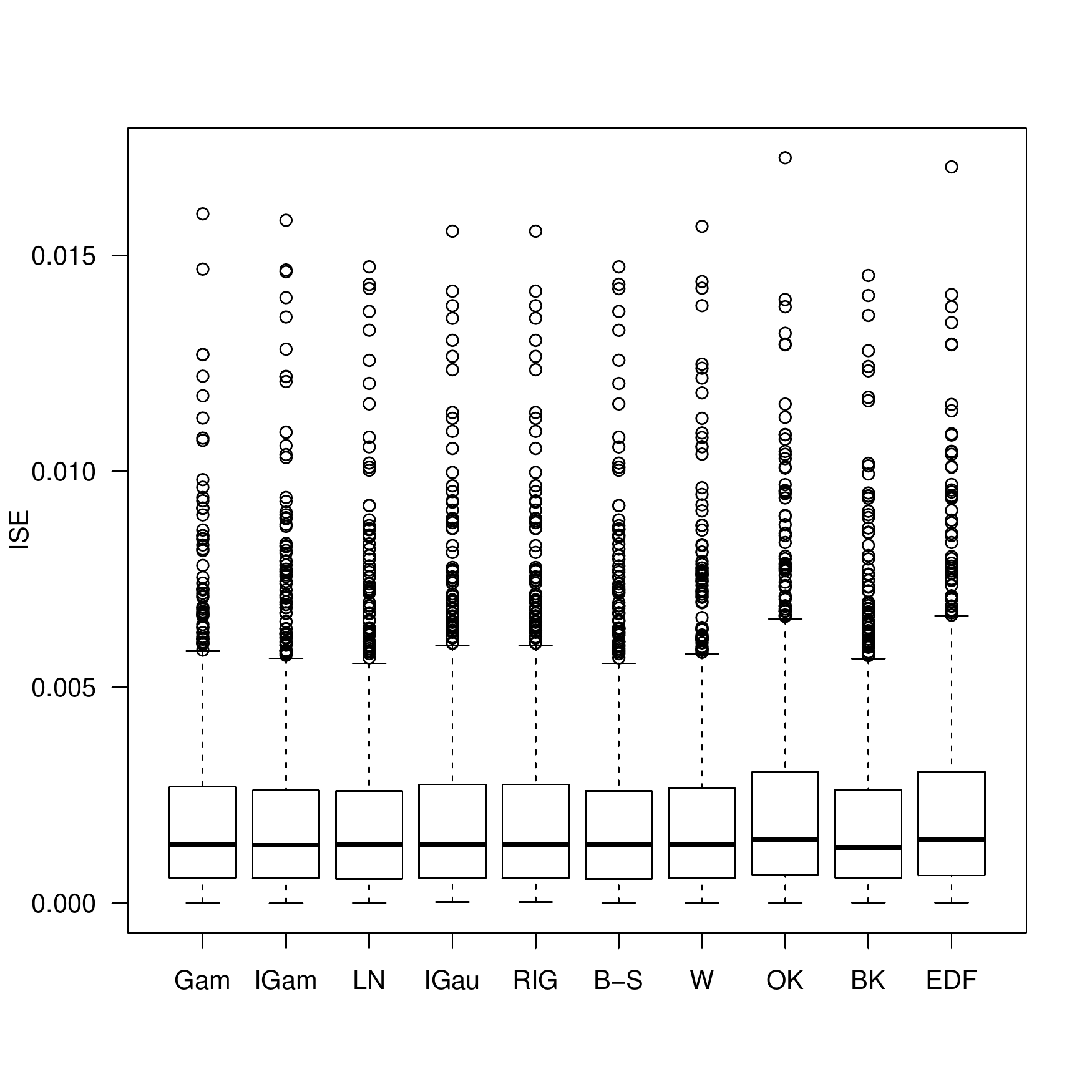}
            \vspace{-0.9cm}
            \caption{$\text{Gamma}\hspace{0.2mm}(4, 2)$}
        \end{subfigure}
        \quad
        \begin{subfigure}[b]{0.40\textwidth}
            \centering
            \includegraphics[width=\textwidth, height=0.85\textwidth]{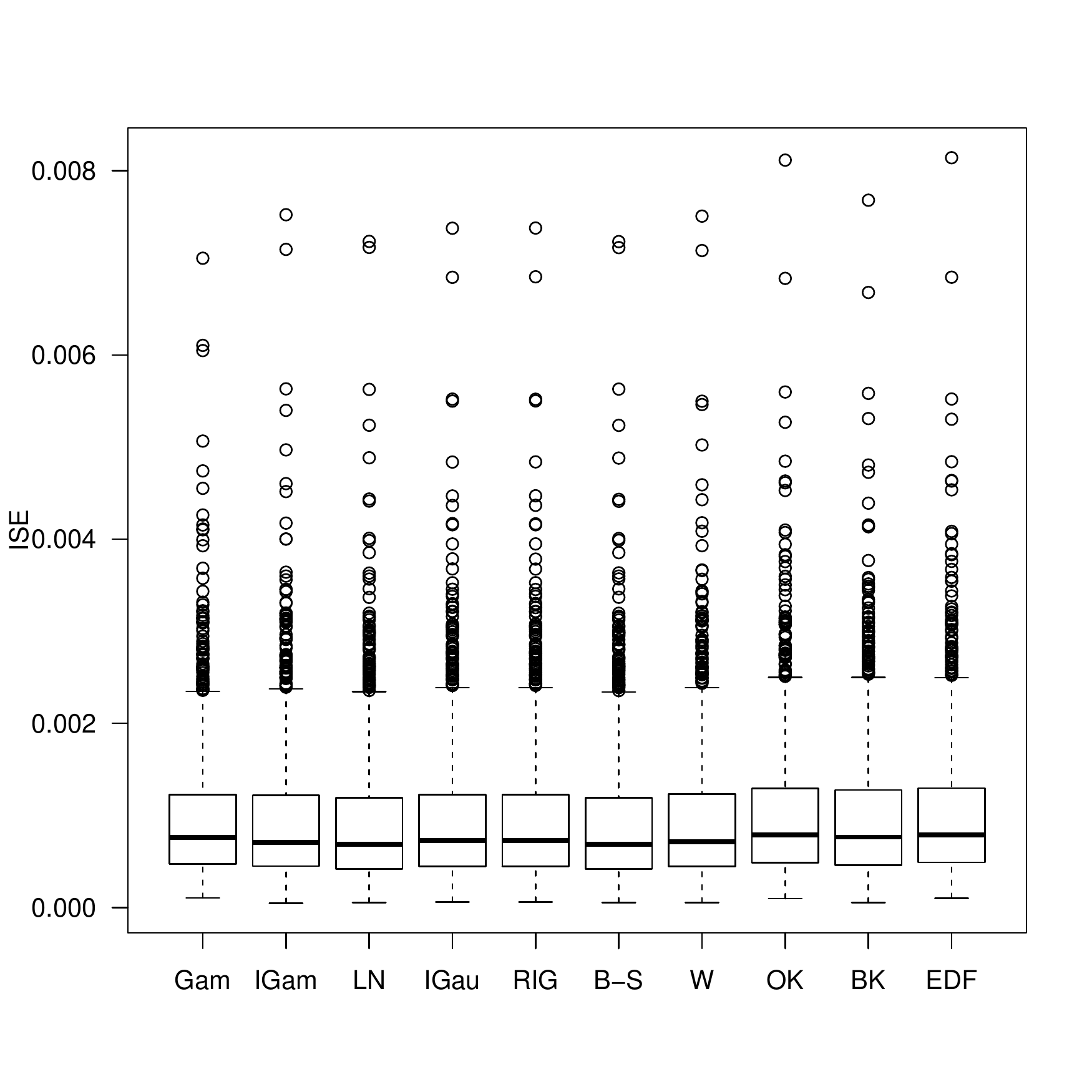}
            \vspace{-0.9cm}
            \caption{$\text{GeneralizedPareto}\hspace{0.2mm}(0.4, 1, 0)$}
        \end{subfigure}
        \begin{subfigure}[b]{0.40\textwidth}
            \centering
            \includegraphics[width=\textwidth, height=0.85\textwidth]{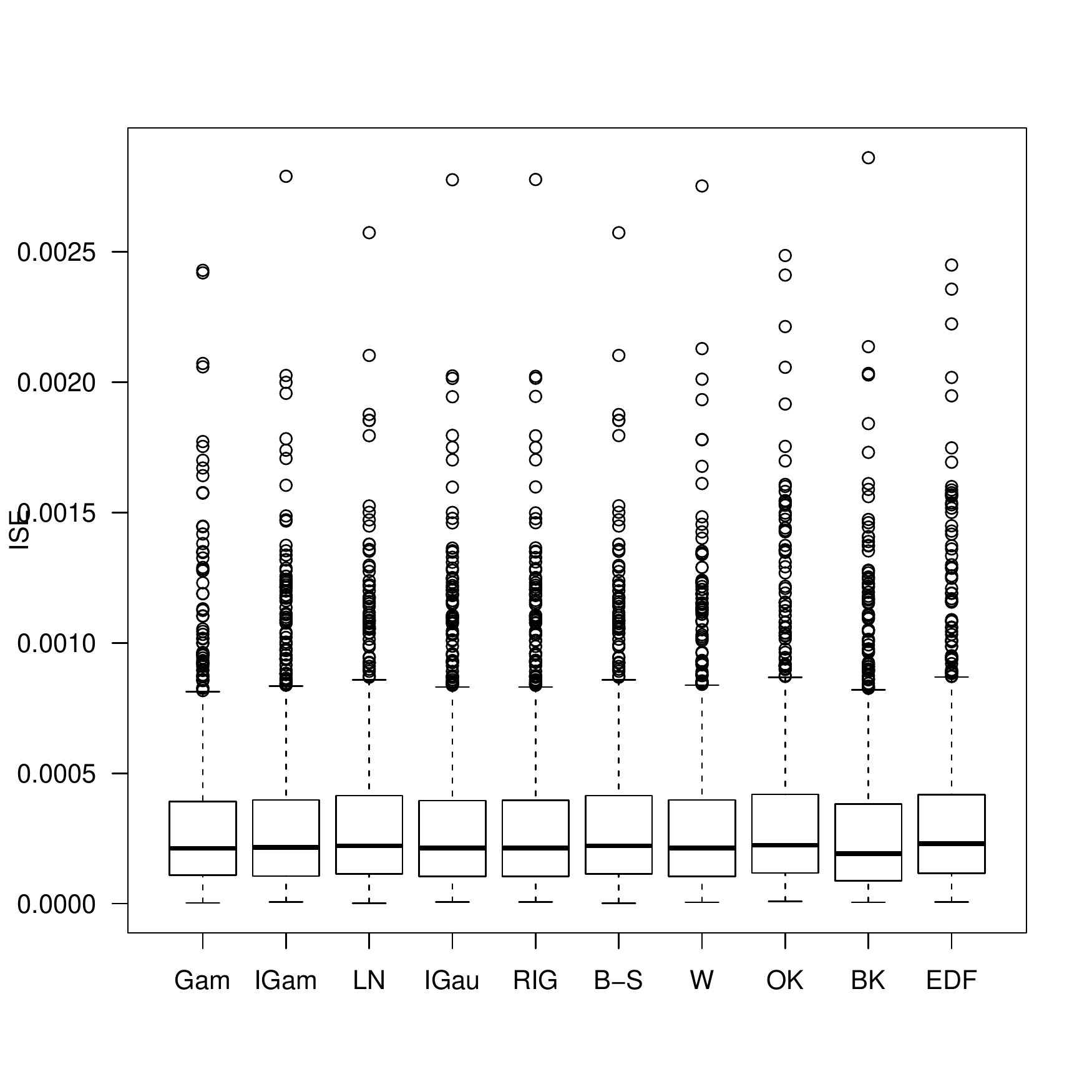}
            \vspace{-0.9cm}
            \caption{$\text{HalfNormal}\hspace{0.2mm}(1)$}
        \end{subfigure}
        \quad
        \begin{subfigure}[b]{0.40\textwidth}
            \centering
            \includegraphics[width=\textwidth, height=0.85\textwidth]{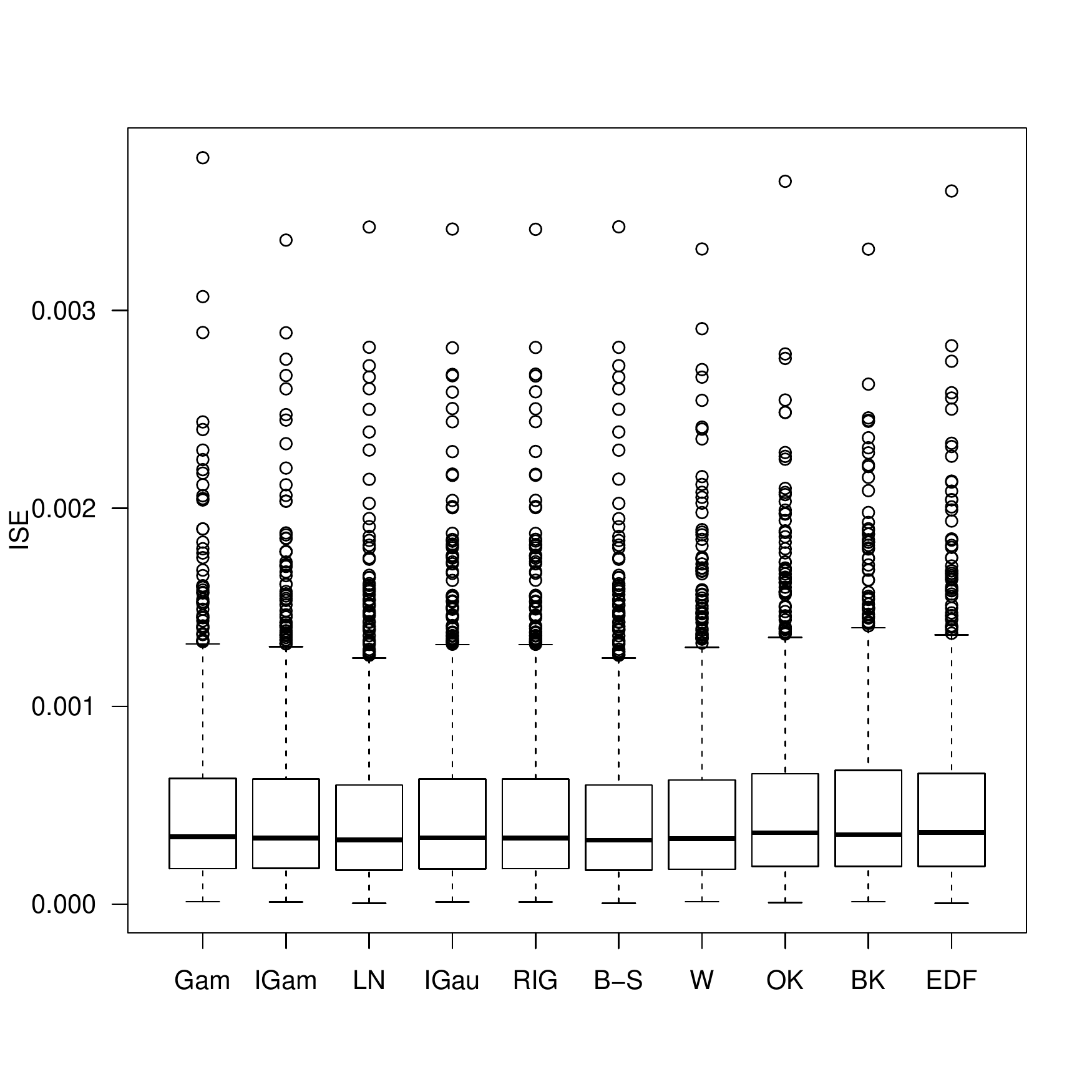}
            \vspace{-0.9cm}
            \caption{$\text{LogNormal}\hspace{0.2mm}(0, 0.75)$}
        \end{subfigure}
        \begin{subfigure}[b]{0.40\textwidth}
            \centering
            \includegraphics[width=\textwidth, height=0.85\textwidth]{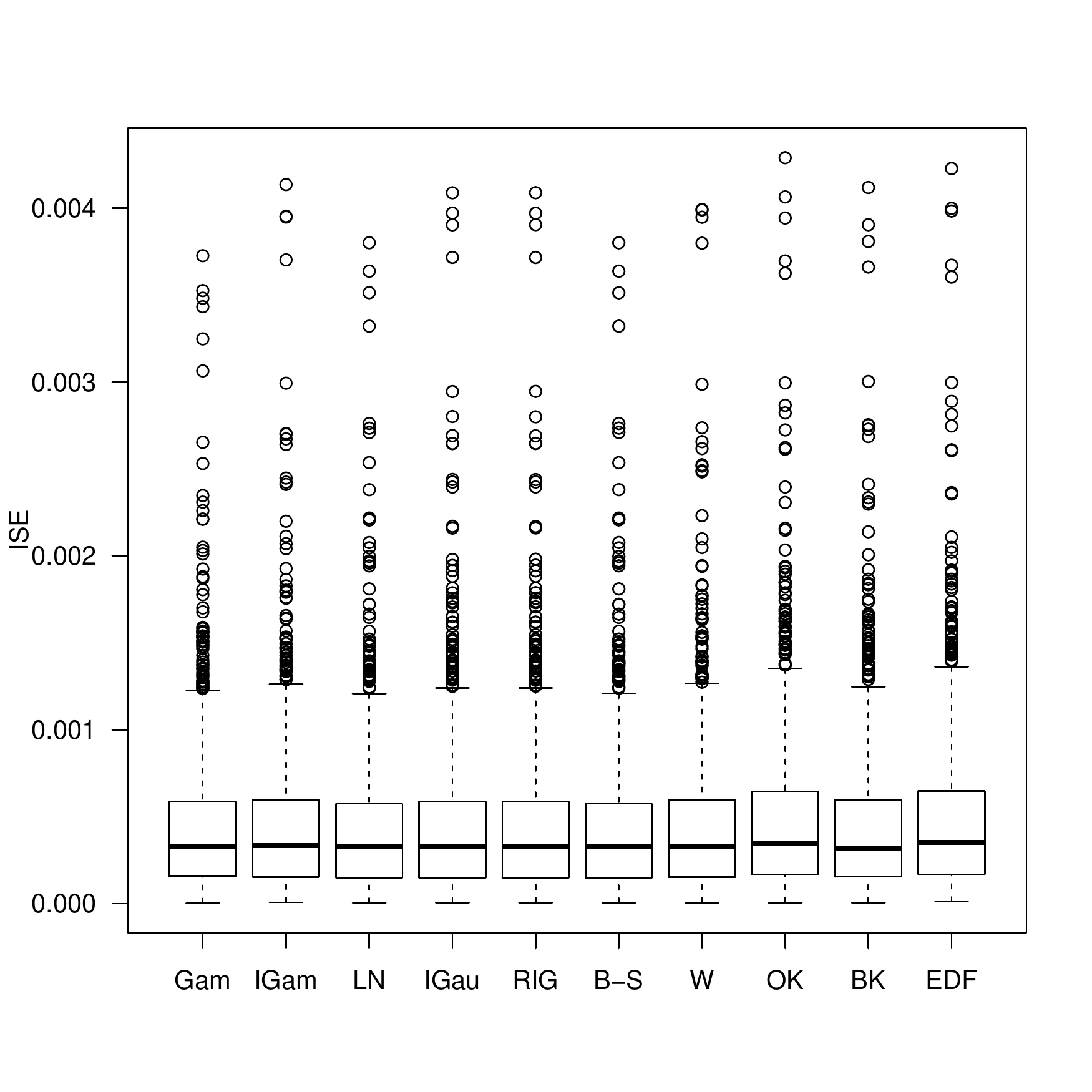}
            \vspace{-0.9cm}
            \caption{$\text{Weibull}\hspace{0.2mm}(1.5, 1.5)$}
        \end{subfigure}
        \quad
        \begin{subfigure}[b]{0.40\textwidth}
            \centering
            \includegraphics[width=\textwidth, height=0.85\textwidth]{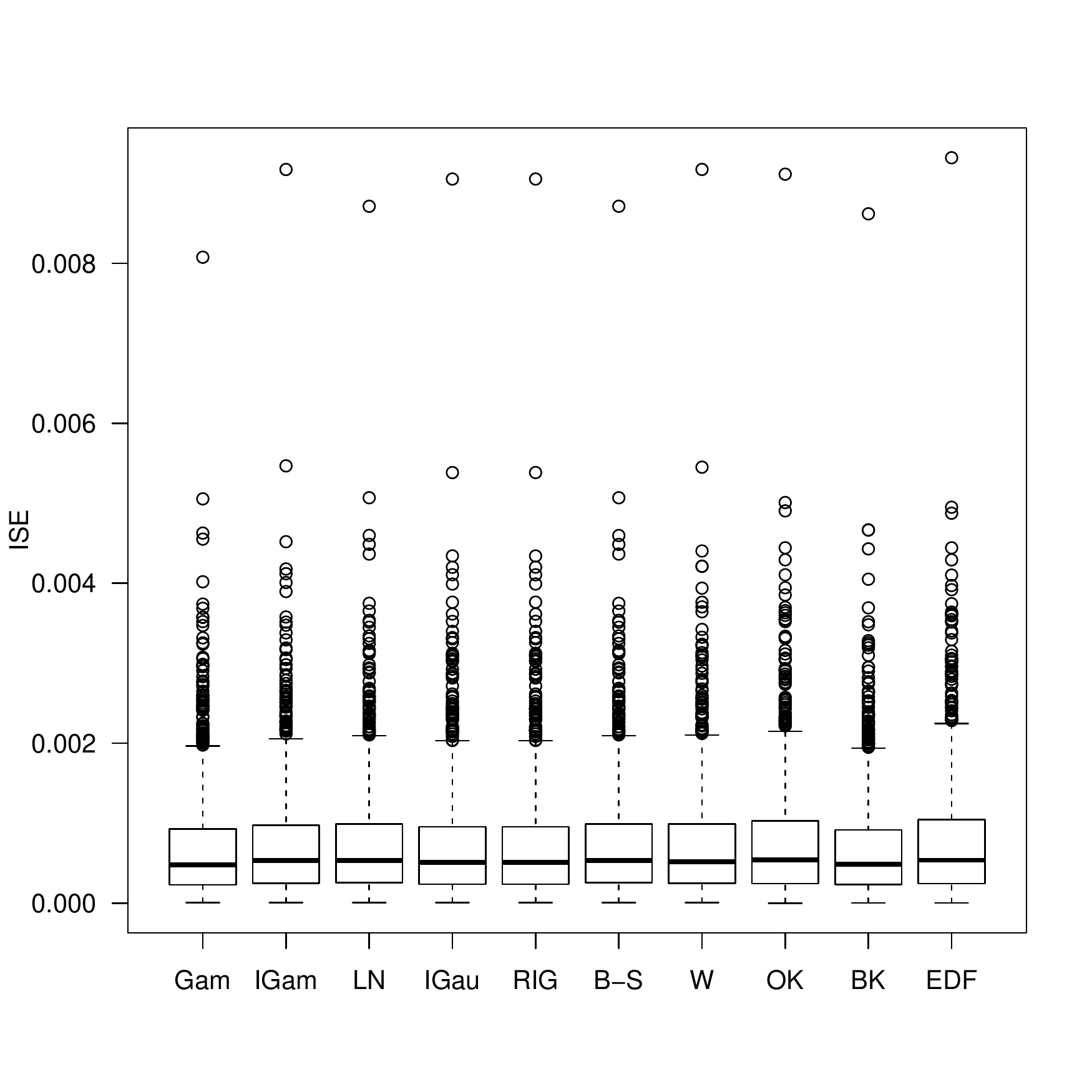}
            \vspace{-0.9cm}
            \caption{$\text{Weibull}\hspace{0.2mm}(3, 2)$}
        \end{subfigure}
        \caption{Boxplots of the $\mathrm{ISE}_{i,j,n}^{(k)}, ~k = 1,2,\dots,M$, for the eight distributions and the ten estimators, when the sample size is $n = 1000$.}
        \label{fig:ISE.boxplots}
    \end{figure}

    \newpage
    \begin{figure}[ht]
        \captionsetup{width=0.8\linewidth}
        \vspace{-0.5cm}
        \centering
        \includegraphics[width=\textwidth, height=1.2\textwidth]{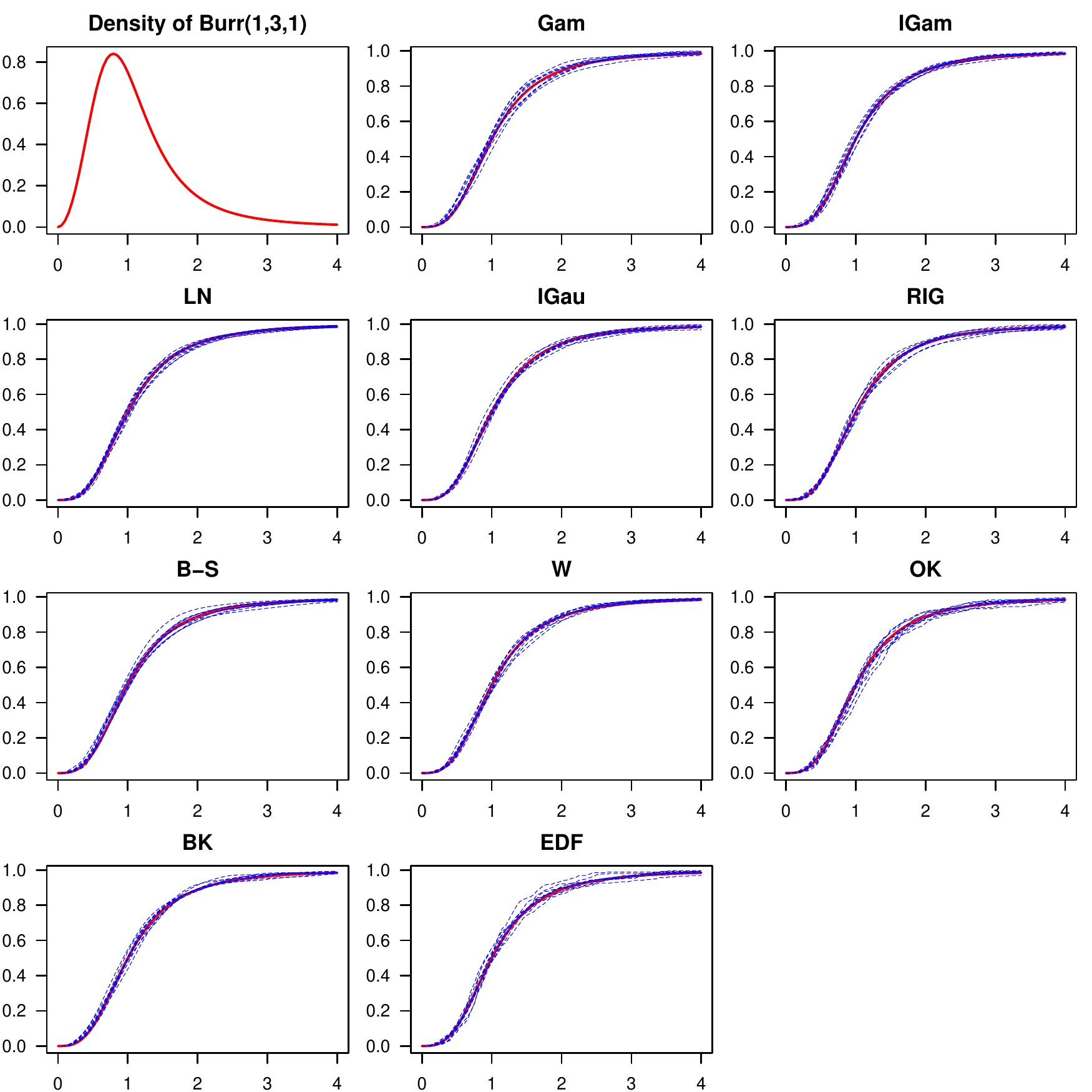}\vspace{-0.25cm}
        \vspace{5mm}
        \caption{The $\text{Burr}\hspace{0.2mm}(1,3,1)$ density function appears on the top-left, and the true c.d.f.\ is depicted {\color{red} in red} everywhere else. Each plot has ten estimates {\color{blue} in blue} for the $\text{Burr}\hspace{0.2mm}(1,3,1)$ c.d.f.\ using one of the ten estimators and $n = 256$.}
        \label{fig:cdf.estimates.burr}
    \end{figure}

    \newpage
    \begin{figure}[ht]
        \captionsetup{width=0.8\linewidth}
        \vspace{-0.5cm}
        \centering
        \includegraphics[width=\textwidth, height=1.2\textwidth]{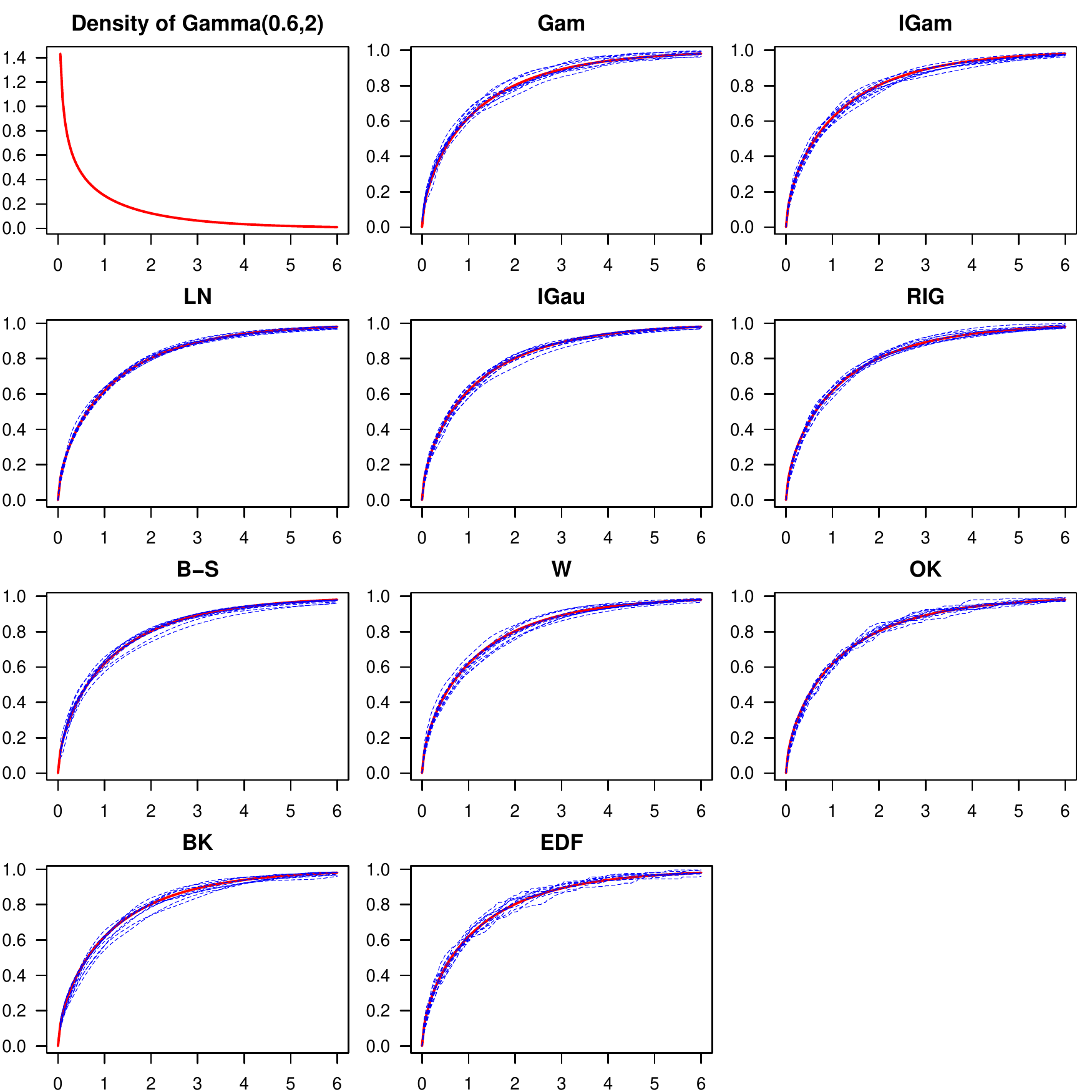}\vspace{-0.25cm}
        \vspace{5mm}
        \caption{The $\text{Gamma}\hspace{0.2mm}(0.6,2)$ density function appears on the top-left, and the true c.d.f.\ is depicted {\color{red} in red} everywhere else. Each plot has ten estimates {\color{blue} in blue} for the $\text{Gamma}\hspace{0.2mm}(0.6,2)$ c.d.f.\ using one of the ten estimators and $n = 256$.}
        \label{fig:cdf.estimates.gamma.1}
    \end{figure}

    \newpage
    \begin{figure}[ht]
        \captionsetup{width=0.8\linewidth}
        \vspace{-0.5cm}
        \centering
        \includegraphics[width=\textwidth, height=1.2\textwidth]{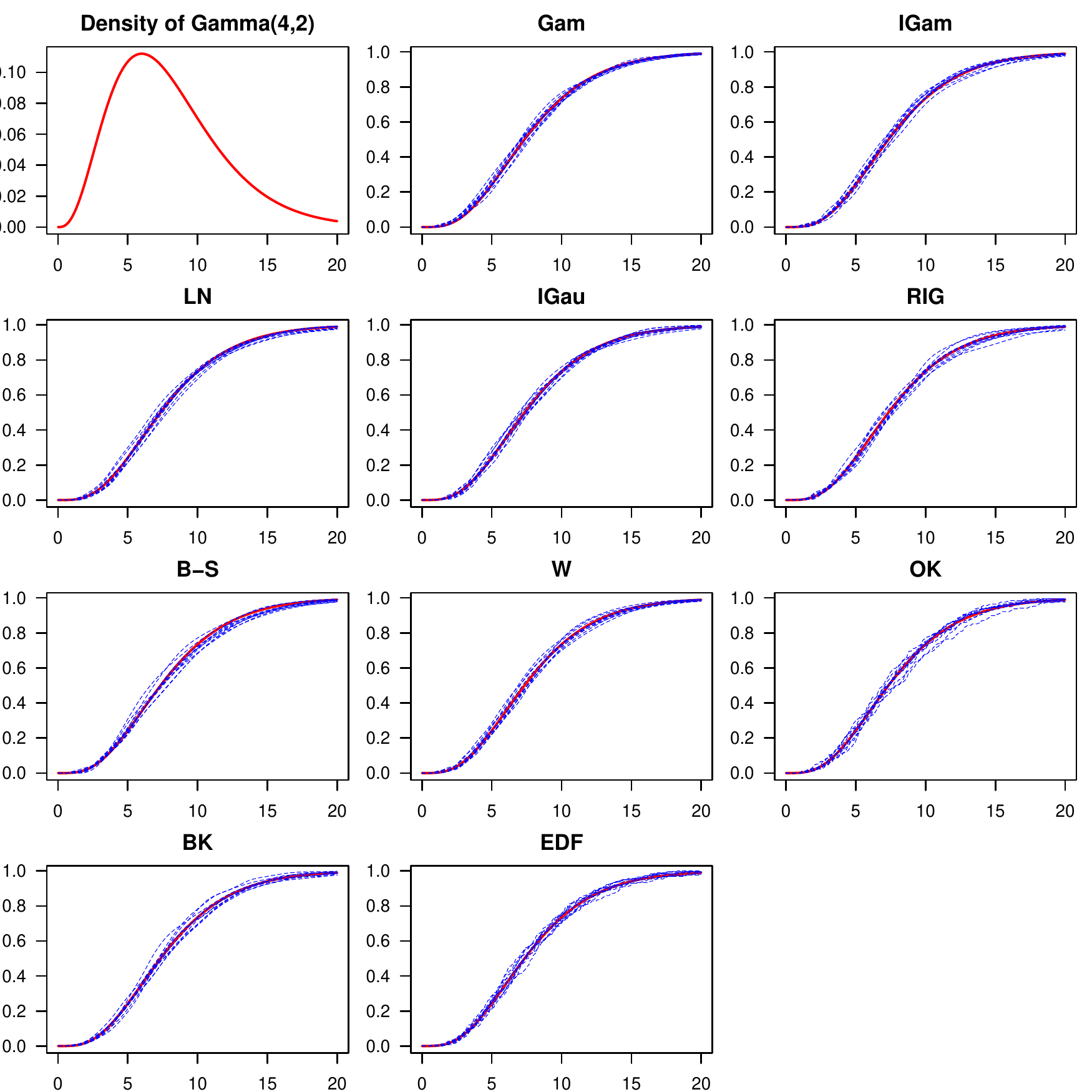}\vspace{-0.25cm}
        \vspace{5mm}
        \caption{The $\text{Gamma}\hspace{0.2mm}(4,2)$ density function appears on the top-left, and the true c.d.f.\ is depicted {\color{red} in red} everywhere else. Each plot has ten estimates {\color{blue} in blue} for the $\text{Gamma}\hspace{0.2mm}(4,2)$ c.d.f.\ using one of the ten estimators and $n = 256$.}
        \label{fig:cdf.estimates.gamma.2}
    \end{figure}

    \newpage
    \begin{figure}[ht]
        \captionsetup{width=0.8\linewidth}
        \centering
        \includegraphics[width=\textwidth, height=1.2\textwidth]{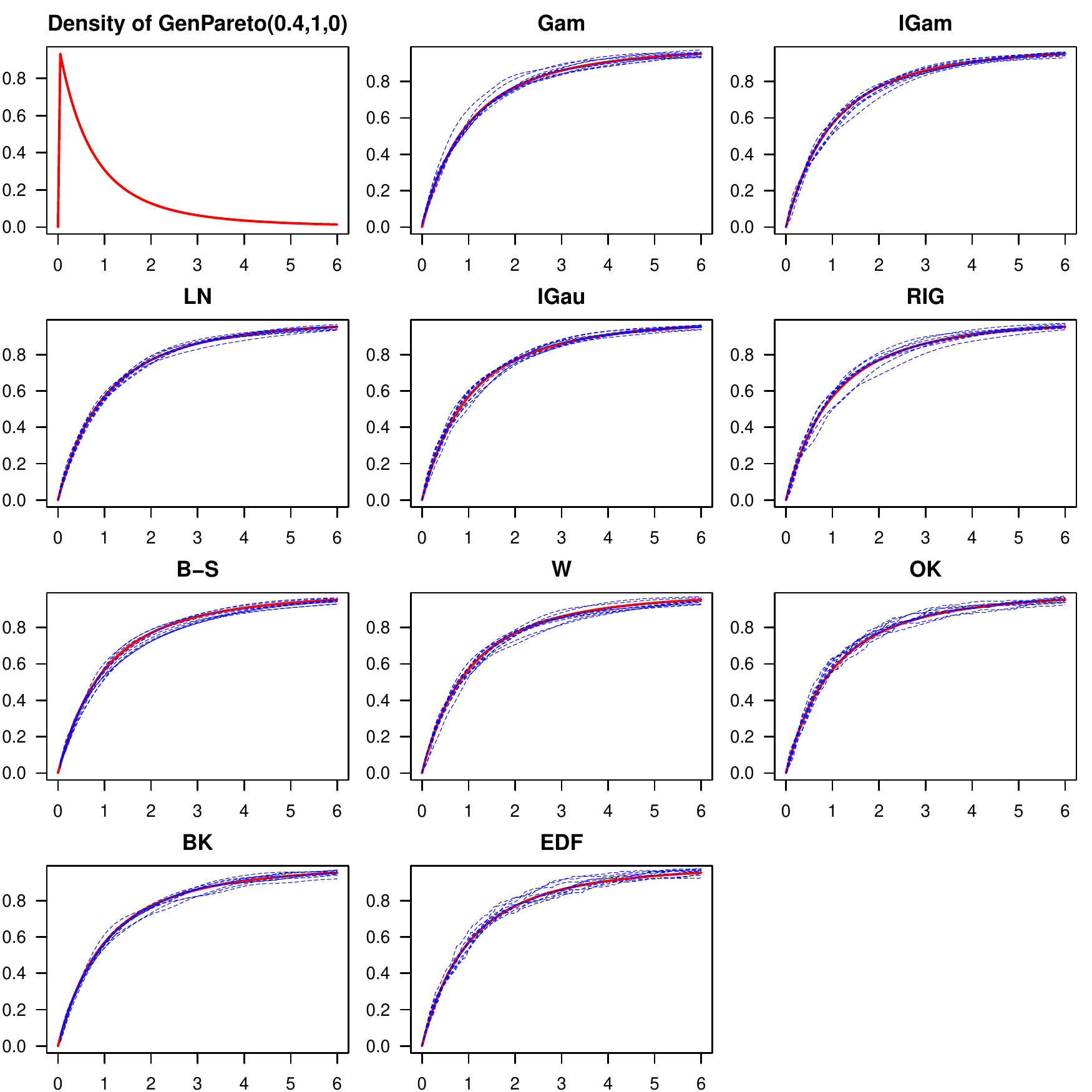}\vspace{-0.25cm}
        \vspace{5mm}
        \caption{The $\text{GeneralizedPareto}\hspace{0.2mm}(0.4,1,0)$ density function appears on the top-left, and the true c.d.f.\ is depicted {\color{red} in red} everywhere else. Each plot has ten estimates {\color{blue} in blue} for the $\text{GeneralizedPareto}\hspace{0.2mm}(0.4,1,0)$ c.d.f.\ using one of the ten estimators and $n = 256$.}
        \label{fig:cdf.estimates.gen.pareto}
    \end{figure}

    \newpage
    \begin{figure}[ht]
        \captionsetup{width=0.8\linewidth}
        \vspace{-0.5cm}
        \centering
        \includegraphics[width=\textwidth, height=1.2\textwidth]{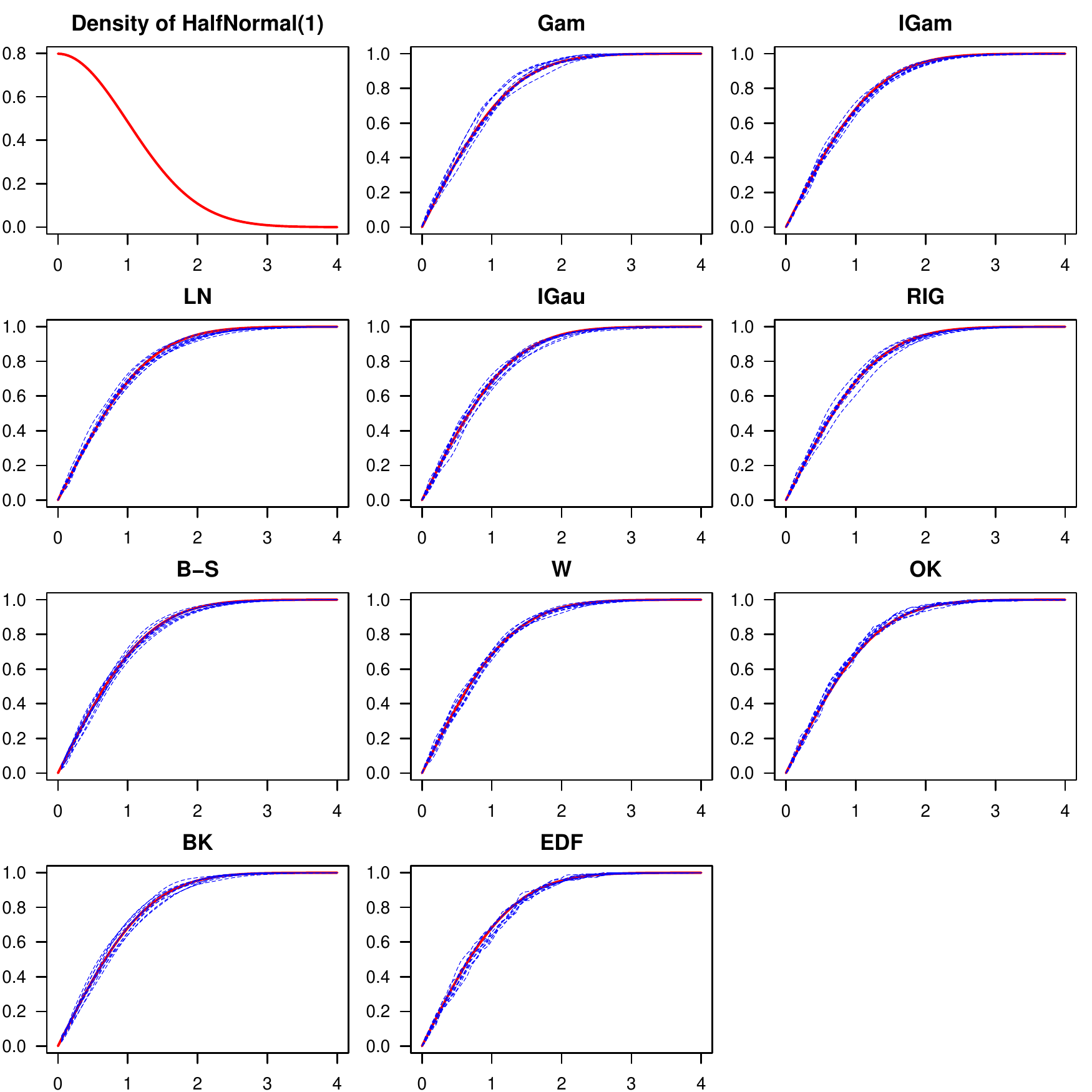}\vspace{-0.25cm}
        \vspace{5mm}
        \caption{The $\text{HalfNormal}\hspace{0.2mm}(1)$ density function appears on the top-left, and the true c.d.f.\ is depicted {\color{red} in red} everywhere else. Each plot has ten estimates {\color{blue} in blue} for the $\text{HalfNormal}\hspace{0.2mm}(1)$ c.d.f.\ using one of the ten estimators and $n = 256$.}
        \label{fig:cdf.estimates.halfnormal}
    \end{figure}

    \newpage
    \begin{figure}[ht]
        \captionsetup{width=0.8\linewidth}
        \vspace{-0.5cm}
        \centering
        \includegraphics[width=\textwidth, height=1.2\textwidth]{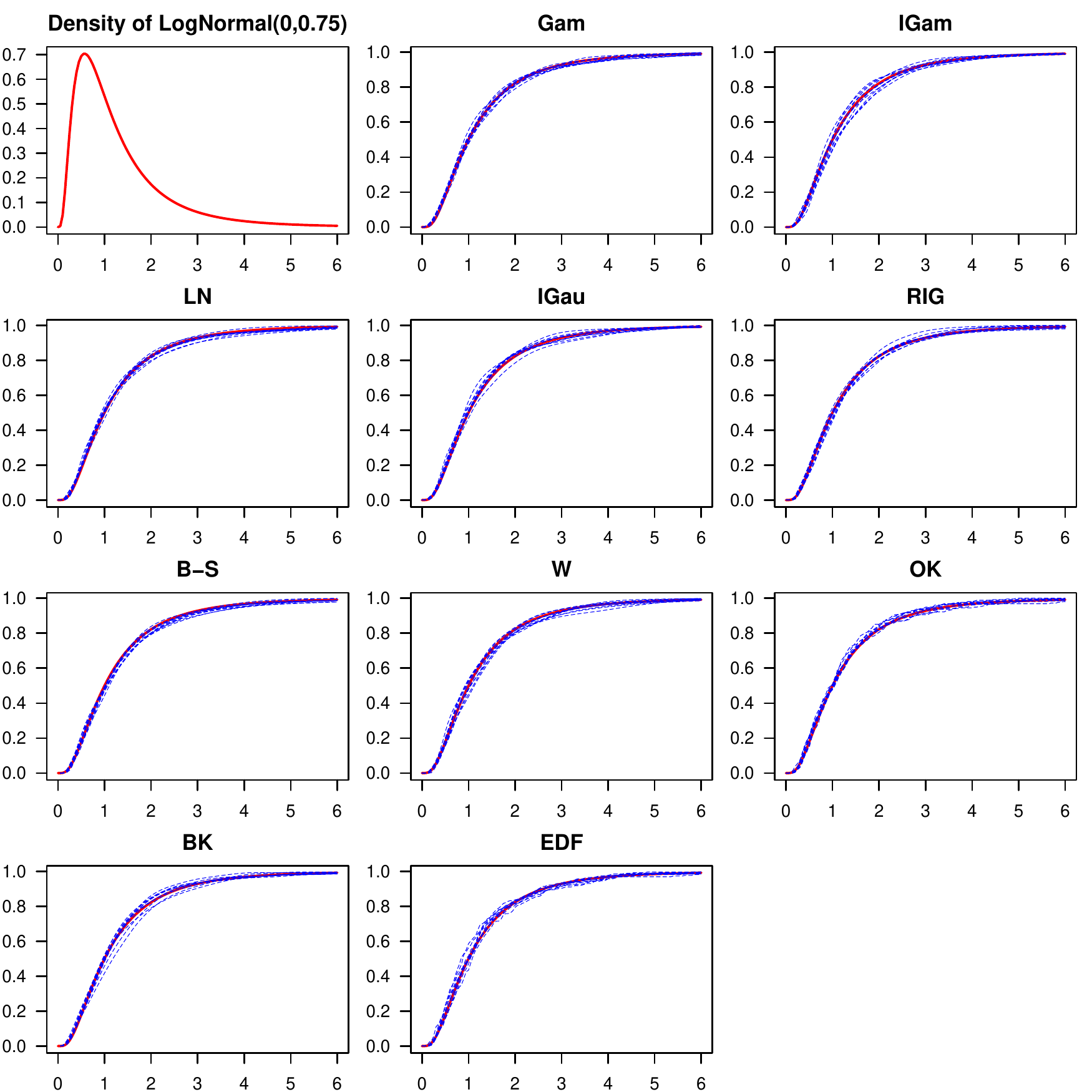}\vspace{-0.25cm}
        \vspace{5mm}
        \caption{The $\text{LogNormal}\hspace{0.2mm}(0,0.75)$ density function appears on the top-left, and the true c.d.f.\ is depicted {\color{red} in red} everywhere else. Each plot has ten estimates {\color{blue} in blue} for the $\text{LogNormal}\hspace{0.2mm}(0,0.75)$ c.d.f.\ using one of the ten estimators and $n = 256$.}
        \label{fig:cdf.estimates.lognormal}
    \end{figure}

    \newpage
    \begin{figure}[ht]
        \captionsetup{width=0.8\linewidth}
        \vspace{-0.5cm}
        \centering
        \includegraphics[width=\textwidth, height=1.2\textwidth]{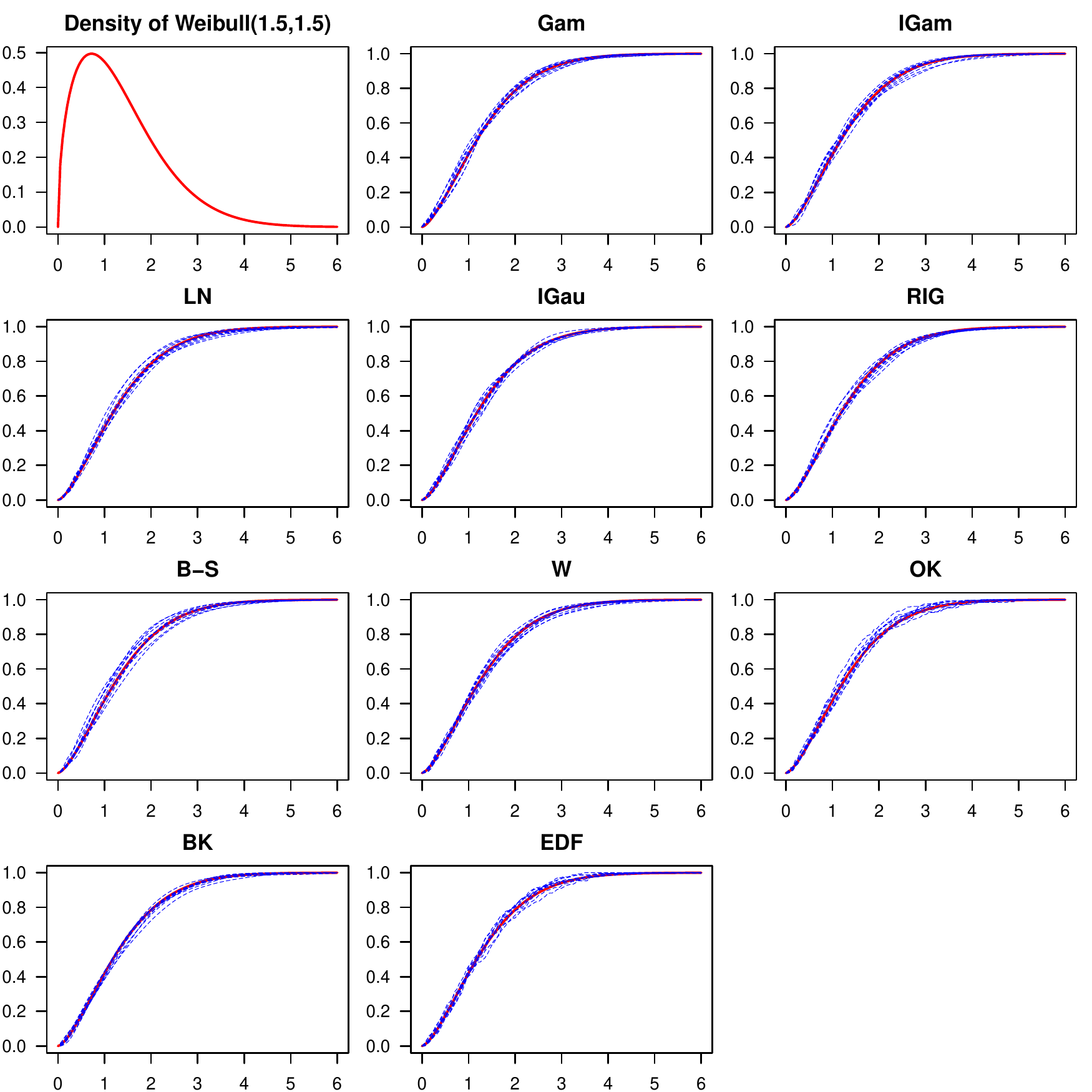}\vspace{-0.25cm}
        \vspace{5mm}
        \caption{The $\text{Weibull}\hspace{0.2mm}(1.5,1.5)$ density function appears on the top-left, and the true c.d.f.\ is depicted {\color{red} in red} everywhere else. Each plot has ten estimates {\color{blue} in blue} for the $\text{Weibull}\hspace{0.2mm}(1.5,1.5)$ c.d.f.\ using one of the ten estimators and $n = 256$.}
        \label{fig:cdf.estimates.weibull.1}
    \end{figure}

    \newpage
    \begin{figure}[ht]
        \captionsetup{width=0.8\linewidth}
        \vspace{-0.5cm}
        \centering
        \includegraphics[width=\textwidth, height=1.2\textwidth]{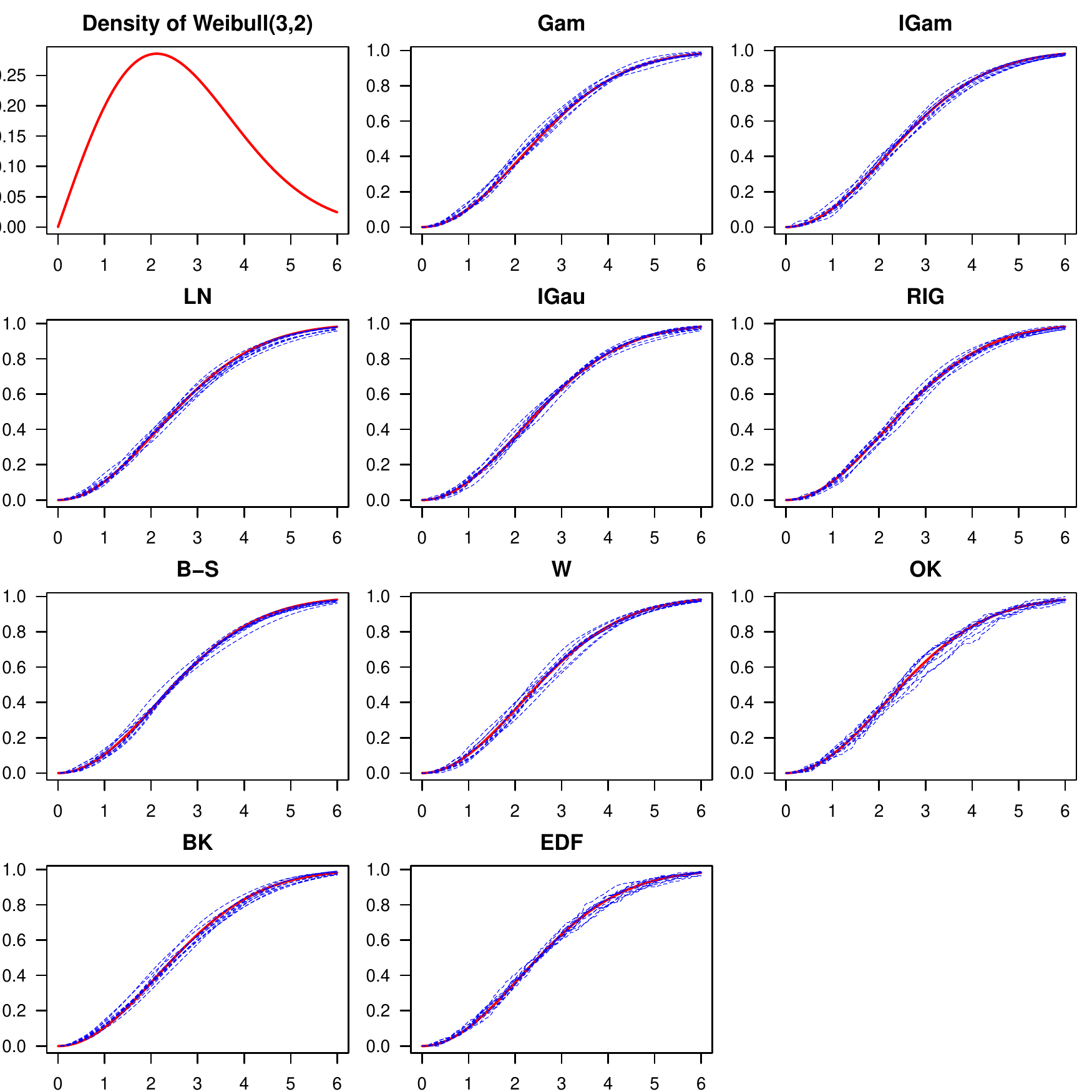}\vspace{-0.25cm}
        \vspace{5mm}
        \caption{The $\text{Weibull}\hspace{0.2mm}(3,2)$ density function appears on the top-left, and the true c.d.f.\ is depicted {\color{red} in red} everywhere else. Each plot has ten estimates {\color{blue} in blue} for the $\text{Weibull}\hspace{0.2mm}(3,2)$ c.d.f.\ using one of the ten estimators and $n = 256$.}
        \label{fig:cdf.estimates.weibull.2}
    \end{figure}

\restoregeometry

\newpage
\section{Discussion of the simulation results}\label{sec:discussion}

In Table~\ref{table:1}, we see that the LN and B-S kernel c.d.f.\ estimators performed the best (had the lowest $\mathrm{ISE}$ means) for the majority of the distributions considered ($j = 1, 2, 3, 4, 6, 7$ for $n = 256$ and $j = 1, 2, 4, 6, 7$ for $n = 1000$).
They also always did so in pair, with the same $\mathrm{ISE}$ mean up to the second decimal.
For the remaining cases, the boundary kernel estimator (BK) from \cite{MR3072469} had the lowest $\mathrm{ISE}$ means.
As expected, the ordinary kernel estimator and the empirical c.d.f.\ performed the worst.
In \cite{Mombeni_et_al_2019_accepted}, the authors reported that the empirical c.d.f.\ performed better than the BK estimator, but this has to be a programming error (especially since the bandwidth was optimized with a plug-in method).
Overall, our means and standard deviations in Table~\ref{table:1} seem to be lower than the ones reported in \cite{Mombeni_et_al_2019_accepted} at least in part because we used a more precise option (the \texttt{pracma::integral} function in \texttt{R}) to approximate the integrals involved in the bandwidth selection procedures and the computation of the $\mathrm{ISE}$'s.

\vspace{2mm}
In all cases, the asymmetric kernel estimators were at least competitive with the BK estimator in Table~\ref{table:1}.
Table~\ref{table:2}, which shows the difference between an estimator's $\mathrm{ISE}$ mean and the lowest $\mathrm{ISE}$ mean for the corresponding distribution and sample size, paints a nicer picture of the asymmetric kernel c.d.f.\ estimators' performance.
It shows that, for each sample size ($n = 256, 1000$), the total of those differences to the best $\mathrm{ISE}$ mean is significantly lower for the LN and B-S kernel estimators, compared to all the other alternatives.
Also, the totals for all the asymmetric kernel estimators are lower than for the BK estimator when $n = 256$ and are in the same range (or better in the case of LN and B-S) when $n = 1000$.
This means that all the asymmetric kernel estimators are overall better alternatives (or at least always remain competitive) compared to the BK estimator, although the advantage seems to dissipate (except for LN and B-S) when $n$ increases.

\section{Conclusion}\label{sec:conclusion}

In this paper, we considered five new asymmetric kernel c.d.f.\ estimators, namely the Gamma (Gam), inverse Gamma (IGam), lognormal (LN), inverse Gaussian (IGau) and reciprocal inverse Gaussian (RIG) kernel c.d.f.\ estimators.
We proved the asymptotic normality of these estimators and also found asymptotic expressions for their bias, variance, mean squared error and mean integrated squared error. The expressions for the optimal bandwidth under the mean integrated squared error was used in each case to implement a bandwidth selection procedure in our simulation study.
With the same experimental design as in \cite{Mombeni_et_al_2019_accepted} (but with an improved approximation of the integrals involved in the bandwidth selection procedures and the computation of the $\mathrm{ISE}$'s), our results show that the lognormal and Birnbaum-Saunders kernel c.d.f.\ estimators perform the best overall. The results also show that all seven asymmetric kernel c.d.f.\ estimator are better in some cases and at least always competitive against the boundary kernel alternative presented in \cite{MR3072469}.
In that sense, all seven asymmetric kernel c.d.f.\ estimators are safe to use in place of more traditional methods.
We recommend using the lognormal and Birnbaum-Saunders kernel c.d.f.\ estimators in the future.

\newpage
\appendix

\begin{appendices}

\section{Proof of the results for the Gam kernel}\label{sec:proof.results.G.kernel}

    \begin{proof}[Proof of Lemma~\ref{lem:bias.variance.G.kernel}]
        If $T$ denotes a random variable with the density
        \begin{equation}
            k_{\mathrm{Gam}}(t \nvert \alpha, \theta) = \frac{t^{\alpha - 1} e^{-t / \theta}}{\theta^{\alpha} \Gamma(\alpha)}, \quad \text{with } (\alpha, \theta) = (b^{-1} x + 1, b),
        \end{equation}
        then integration by parts yields
        \begin{align}
            \EE[\hat{F}_{n,b}^{\mathrm{Gam}}(x)] - F(x)
            &= \EE[F(T)] - F(x) \notag \\
            &= f(x) \cdot \EE[T - x] + \frac{1}{2} f'(x) \cdot \EE[(T - x)^2] + \oo_x\big(\EE[(T - x)^2]\big) \notag \\[-1mm]
            &= f(x) \cdot b + \frac{1}{2} f'(x) \cdot b (2 b + x) + \oo_x(b) \notag \\
            &= b \cdot (f(x) + \frac{x}{2} f'(x)) + \oo_x(b).
        \end{align}
        Now, we want to compute the expression for the variance.
        Let $S$ be a random variable with density $t\mapsto 2 k_{\mathrm{Gam}}(t \nvert b^{-1} x + 1, b) \overline{K}_{\mathrm{Gam}}(t \nvert b^{-1} x + 1, b)$  and note that $\min\{T_1,T_2\}$ has that particular distribution if $T_1,T_2\sim \mathrm{Gam}(b^{-1} x + 1, b)$ are independent.
        Then, integration by parts and Corollary~\ref{cor:tech.var.G} yield, for any given $x\in (0,\infty)$,
        \begin{align}
            \EE\big[\overline{K}_{\mathrm{Gam}}^2(X_1, b^{-1} x + 1, b)\big]
            &= \EE[F(S)] \notag \\[0.5mm]
            &= F(x) + f(x) \cdot \EE[S - x] + \OO_x\big(\EE[(S - x)^2]\big) \notag \\[-1mm]
            &= F(x) + f(x) \cdot \big[-\sqrt{\frac{b x}{\pi}} + \OO_x(b)\big] + \OO_x(b) \notag \\[-1mm]
            &= F(x) - b^{1/2} \cdot \frac{\sqrt{x} f(x)}{\sqrt{\pi}} + \OO_x(b).
        \end{align}
        so that
        \begin{align}
            \VV(\hat{F}_{n,b}^{\mathrm{Gam}}(x))
            &= n^{-1} \EE\big[\overline{K}_{\mathrm{Gam}}^2(X_1, b^{-1} x + 1, b)\big] - n^{-1} \big(\EE[\hat{F}_{n,b}^{\mathrm{Gam}}(x)]\big)^2 \notag \\
            &= n^{-1} F(x) (1 - F(x)) - n^{-1} b^{1/2} \cdot \frac{\sqrt{x} f(x)}{\sqrt{\pi}} + \OO_x(n^{-1} b).
        \end{align}
        This ends the proof.
    \end{proof}

    \begin{proof}[Proof of Proposition~\ref{prop:MISE.G.kernel}]
        Note that $\hat{F}_{n,b}^{\mathrm{Gam}}(x) - \EE[\hat{F}_{n,b}^{\mathrm{Gam}}(x)] = \frac{1}{n} \sum_{i=1}^n Z_{i,b}$ where
        \begin{equation}
            Z_{i,b} \leqdef \overline{K}_{\mathrm{Gam}}(X_i \nvert b^{-1} x + 1, b) - \EE[\overline{K}_{\mathrm{Gam}}(X_i \nvert b^{-1} x + 1, b)], \quad 1 \leq i \leq n,
        \end{equation}
        are i.i.d.\ and centered random variables.
        It suffices to show the following Lindeberg condition for double arrays (see, e.g., Section 1.9.3 in \cite{MR0595165}):
        For every $\varepsilon > 0$,
        \begin{equation}
            s_b^{-2} \, \EE\big[Z_{1,b}^2 \ind_{\{|Z_{1,b}| > \varepsilon  n^{1/2} s_b\}}\big] \longrightarrow 0, \quad \text{as } n\to \infty,
        \end{equation}
        where $s_b^2 \leqdef \EE[Z_{1,b}^2]$ and $b = b(n)\to 0$.
        But this follows from the fact that $|Z_{1,b}| \leq 2$ for all $b > 0$, and $s_b = (n\hspace{0.3mm} \VV(\hat{F}_{n,b}^{\mathrm{Gam}}))^{1/2}\to F(x) (1 - F(x))$ as $n\to \infty$ by Lemma~\ref{lem:bias.variance.G.kernel}.
    \end{proof}

\section{Proof of the results for the IGam kernel}\label{sec:proof.results.IGam.kernel}

    \begin{proof}[Proof of Lemma~\ref{lem:bias.variance.IGam.kernel}]
        If $T$ denotes a random variable with the density
        \begin{equation}
            k_{\mathrm{IGam}}(t \nvert \alpha, \theta) = \frac{t^{-\alpha - 1} e^{-1 / (t \theta)}}{\theta^{\alpha} \Gamma(\alpha)}, \quad \text{with } (\alpha, \theta) = (b^{-1} + 1, x^{-1} b),
        \end{equation}
        then integration by parts yields (assuming $0 < b < 1/2$)
        \begin{align}
            \EE[\hat{F}_{n,b}^{\mathrm{IGam}}(x)] - F(x)
            &= \EE[F(T)] - F(x) \notag \\[0.5mm]
            &= f(x) \cdot \EE[T - x] + \frac{1}{2} f'(x) \cdot \EE[(T - x)^2] + \oo_x\big(\EE[(T - x)^2]\big) \notag \\[-1mm]
            &= f(x) \cdot 0 + \frac{1}{2} f'(x) \cdot \frac{b x^2}{1 - b} + \oo_x(b) \notag \\[-2mm]
            &= b \cdot \frac{x^2}{2} f'(x) + \oo_x(b).
        \end{align}
        Now, we want to compute the expression for the variance.
        Let $S$ be a random variable with density $t\mapsto 2 k_{\mathrm{IGam}}(t \nvert b^{-1} + 1, x^{-1} b) \overline{K}_{\mathrm{IGam}}(t \nvert b^{-1} + 1, x^{-1} b)$  and note that $\min\{T_1,T_2\}$ has that particular distribution if $T_1,T_2\sim \mathrm{IGam}(b^{-1} + 1, x^{-1} b)$ are independent.
        Then, integration by parts and Corollary~\ref{cor:tech.var.IGam}, for any given $x\in (0,\infty)$,
        \begin{align}
            \EE\big[\overline{K}_{\mathrm{IGam}}^2(X_1, b^{-1} + 1, x^{-1} b)\big]
            &= \EE[F(S)] \notag \\[0.5mm]
            &= F(x) + f(x) \cdot \EE[S - x] + \OO_x\big(\EE[(S - x)^2]\big) \notag \\[-1mm]
            &= F(x) + f(x) \cdot \big[-x \sqrt{\frac{b}{\pi}} + \OO_x(b)\big] + \OO_x(b) \notag \\[-1.5mm]
            &= F(x) - b^{1/2} \cdot \frac{x f(x)}{\sqrt{\pi}} + \OO_x(b).
        \end{align}
        so that
        \begin{align}
            \VV(\hat{F}_{n,b}^{\mathrm{IGam}}(x))
            &= n^{-1} \EE\big[\overline{K}_{\mathrm{IGam}}^2(X_1, b^{-1} + 1, x^{-1} b)\big] - n^{-1} \big(\EE[\hat{F}_{n,b}^{\mathrm{IGam}}(x)]\big)^2 \notag \\
            &= n^{-1} F(x) (1 - F(x)) - n^{-1} b^{1/2} \cdot \frac{x f(x)}{\sqrt{\pi}} + \OO_x(n^{-1} b).
        \end{align}
        This ends the proof.
    \end{proof}

    \begin{proof}[Proof of Proposition~\ref{prop:MISE.IGam.kernel}]
        Note that $\hat{F}_{n,b}^{\mathrm{IGam}}(x) - \EE[\hat{F}_{n,b}^{\mathrm{IGam}}(x)] = \frac{1}{n} \sum_{i=1}^n Z_{i,b}$ where
        \begin{equation}
            Z_{i,b} \leqdef \overline{K}_{\mathrm{IGam}}(X_i \nvert b^{-1} + 1, x^{-1} b) - \EE[\overline{K}_{\mathrm{IGam}}(X_i \nvert b^{-1} + 1, x^{-1} b)], \quad 1 \leq i \leq n,
        \end{equation}
        are i.i.d.\ and centered random variables.
        It suffices to show the following Lindeberg condition for double arrays (see, e.g., Section 1.9.3 in \cite{MR0595165}):
        For every $\varepsilon > 0$,
        \begin{equation}
            s_b^{-2} \, \EE\big[Z_{1,b}^2 \ind_{\{|Z_{1,b}| > \varepsilon  n^{1/2} s_b\}}\big] \longrightarrow 0, \quad \text{as } n\to \infty,
        \end{equation}
        where $s_b^2 \leqdef \EE[Z_{1,b}^2]$ and $b = b(n)\to 0$.
        But this follows from the fact that $|Z_{1,b}| \leq 2$ for all $b > 0$, and $s_b = (n\hspace{0.3mm} \VV(\hat{F}_{n,b}^{\mathrm{IGam}}))^{1/2}\to F(x) (1 - F(x))$ as $n\to \infty$ by Lemma~\ref{lem:bias.variance.IGam.kernel}.
    \end{proof}

\section{Proof of the results for the LN kernel}\label{sec:proof.results.LN.kernel}

    \begin{proof}[Proof of Lemma~\ref{lem:bias.variance.LN.kernel}]
        If $T$ denotes a random variable with the density
        \begin{equation}
            k_{\mathrm{LN}}(t \nvert \mu, \sigma) = \frac{1}{t \sqrt{2\pi \sigma^2}} \exp\bigg(-\frac{(\log t - \mu)^2}{2 \sigma^2}\bigg), \quad \text{with } (\mu, \sigma) = (\log x, \sqrt{b}),
        \end{equation}
        then integration by parts yields
        \begin{align}
            \EE[\hat{F}_{n,b}^{\mathrm{LN}}(x)] - F(x)
            &= \EE[F(T)] - F(x) \notag \\
            &= f(x) \cdot \EE[T - x] + \frac{1}{2} f'(x) \cdot \EE[(T - x)^2] + \oo_x\big(\EE[(T - x)^2]\big) \notag \\[-1mm]
            &= f(x) \cdot x (e^{b/2} - 1) + \frac{1}{2} f'(x) \cdot x^2 (e^{2b} - 2 e^{b/2} + 1) + \oo_x(b) \notag \\
            &= b \cdot \frac{x}{2} (f(x) + x f'(x)) + \oo_x(b).
        \end{align}
        Now, we want to compute the expression for the variance.
        Let $S$ be a random variable with density $t\mapsto 2 k_{\mathrm{LN}}(t \nvert \log x, \sqrt{b}) \overline{K}_{\mathrm{LN}}(t \nvert \log x, \sqrt{b})$  and note that $\min\{T_1,T_2\}$ has that particular distribution if $T_1,T_2\sim \mathrm{LN}(\log x, \sqrt{b})$ are independent.
        Then, integration by parts and Corollary~\ref{cor:tech.var.LN} yield, for any given $x\in (0,\infty)$,
        \begin{align}
            \EE\big[\overline{K}_{\mathrm{LN}}^2(X_1, \log x, \sqrt{b})\big]
            &= \EE[F(S)] \notag \\[0.5mm]
            &= F(x) + f(x) \cdot \EE[S - x] + \OO_x\big(\EE[(S - x)^2]\big) \notag \\[0mm]
            &= F(x) + f(x) \cdot \big[b^{1/2} \cdot \frac{-x}{\sqrt{\pi}} + \OO_x(b)\big] + \OO_x(b) \notag \\[-1.5mm]
            &= F(x) - b^{1/2} \cdot \frac{x f(x)}{\sqrt{\pi}} + \OO_x(b).
        \end{align}
        so that
        \begin{align}
            \VV(\hat{F}_{n,b}^{\mathrm{LN}}(x))
            &= n^{-1} \EE\big[\overline{K}_{\mathrm{LN}}^2(X_1, \log x, \sqrt{b})\big] - n^{-1} \big(\EE[\hat{F}_{n,b}^{\mathrm{LN}}(x)]\big)^2 \notag \\
            &= n^{-1} F(x) (1 - F(x)) - n^{-1} b^{1/2} \cdot \frac{x f(x)}{\sqrt{\pi}} + \OO_x(n^{-1} b).
        \end{align}
        This ends the proof.
    \end{proof}

    \begin{proof}[Proof of Proposition~\ref{prop:MISE.LN.kernel}]
        Note that $\hat{F}_{n,b}^{\mathrm{LN}}(x) - \EE[\hat{F}_{n,b}^{\mathrm{LN}}(x)] = \frac{1}{n} \sum_{i=1}^n Z_{i,b}$ where
        \begin{equation}
            Z_{i,b} \leqdef \overline{K}_{\mathrm{LN}}(X_i \nvert \log x, \sqrt{b}) - \EE[\overline{K}_{\mathrm{LN}}(X_i \nvert \log x, \sqrt{b})], \quad 1 \leq i \leq n,
        \end{equation}
        are i.i.d.\ and centered random variables.
        It suffices to show the following Lindeberg condition for double arrays (see, e.g., Section 1.9.3 in \cite{MR0595165}):
        For every $\varepsilon > 0$,
        \begin{equation}
            s_b^{-2} \, \EE\big[Z_{1,b}^2 \ind_{\{|Z_{1,b}| > \varepsilon  n^{1/2} s_b\}}\big] \longrightarrow 0, \quad \text{as } n\to \infty,
        \end{equation}
        where $s_b^2 \leqdef \EE[Z_{1,b}^2]$ and $b = b(n)\to 0$.
        But this follows from the fact that $|Z_{1,b}| \leq 2$ for all $b > 0$, and $s_b = (n\hspace{0.3mm} \VV(\hat{F}_{n,b}^{\mathrm{LN}}))^{1/2}\to F(x) (1 - F(x))$ as $n\to \infty$ by Lemma~\ref{lem:bias.variance.LN.kernel}.
    \end{proof}

\section{Proof of the results for the IGau kernel}\label{sec:proof.results.IGau.kernel}

    \begin{proof}[Proof of Lemma~\ref{lem:bias.variance.IGau.kernel}]
        If $T$ denotes a random variable with the density
        \begin{equation}
            k_{\mathrm{IGau}}(t \nvert \mu, \lambda) = \sqrt{\frac{\lambda}{2\pi t^3}} \exp\bigg(- \lambda \frac{(t - \mu)^2}{2 \mu^2 t}\bigg), \quad \text{with } (\mu,\lambda) = (x, b^{-1} x),
        \end{equation}
        then integration by parts yields
        \begin{align}
            \EE[\hat{F}_{n,b}^{\mathrm{IGau}}(x)] - F(x)
            &= \EE[F(T)] - F(x) \notag \\
            &= f(x) \cdot \EE[T - x] + \frac{1}{2} f'(x) \cdot \EE[(T - x)^2] + \oo_x\big(\EE[(T - x)^2]\big) \notag \\
            &= f(x) \cdot 0 + \frac{1}{2} f'(x) \cdot x^2 b + \oo_x(b) \notag \\[-0.8mm]
            &= b \cdot \frac{x^2}{2} f'(x) + \oo_x(b).
        \end{align}
        Now, we want to compute the expression for the variance.
        Let $S$ be a random variable with density $t\mapsto 2 k_{\mathrm{IGau}}(t \nvert x, b^{-1} x) \overline{K}_{\mathrm{IGau}}(t \nvert x, b^{-1} x)$ and note that $\min\{T_1,T_2\}$, which can also be written as $\frac{1}{2} (T_1 + T_2) - \frac{1}{2} |T_1 - T_2|$, has that particular distribution if $T_1,T_2\sim \mathrm{IGau}(x, b^{-1} x)$ are independent.
        Then, integration by parts together with the fact that $\EE[T_1] = \EE[T_2] = x$ yield, for any given $x\in (0,\infty)$,
        \begin{align}
            \EE\big[\overline{K}_{\mathrm{IGau}}^2(X_1, x, b^{-1} x)\big]
            &= \EE[F(S)] \notag \\[1mm]
            &= F(x) + f(x) \cdot \EE[S - x] + \OO_x\big(\EE[(S - x)^2]\big) \notag \\
            &= F(x) + f(x) \cdot \Big\{-\frac{1}{2} \EE[|T_1 - T_2|]\Big\} + \OO_x(b) \notag \\
            &= F(x) - b^{1/2} \cdot \frac{f(x)}{2} \Big[\lim_{b\to 0} b^{-1/2} \, \EE[|T_1 - T_2|]\Big] + \OO_x(b).
        \end{align}
        so that
        \begin{align}
            \VV(\hat{F}_{n,b}^{\mathrm{IGau}}(x))
            &= n^{-1} \EE\big[\overline{K}_{\mathrm{IGau}}^2(X_1, x, b^{-1} x)\big] - n^{-1} \big(\EE[\hat{F}_{n,b}^{\mathrm{IGau}}(x)]\big)^2 \notag \\[1.5mm]
            &= n^{-1} F(x) (1 - F(x)) \notag \\
            &\quad- n^{-1} b^{1/2} \cdot \frac{f(x)}{2} \Big[\lim_{b\to 0} b^{-1/2} \, \EE[|T_1 - T_2|]\Big] + \OO_x(n^{-1} b).
        \end{align}
        This ends the proof.
    \end{proof}

    \begin{proof}[Proof of Proposition~\ref{prop:MISE.IGau.kernel}]
        Note that $\hat{F}_{n,b}^{\mathrm{IGau}}(x) - \EE[\hat{F}_{n,b}^{\mathrm{IGau}}(x)] = \frac{1}{n} \sum_{i=1}^n Z_{i,b}$ where
        \begin{equation}
            Z_{i,b} \leqdef \overline{K}_{\mathrm{IGau}}(X_i \nvert x, b^{-1} x) - \EE[\overline{K}_{\mathrm{IGau}}(X_i \nvert x, b^{-1} x)], \quad 1 \leq i \leq n,
        \end{equation}
        are i.i.d.\ and centered random variables.
        It suffices to show the following Lindeberg condition for double arrays (see, e.g., Section 1.9.3 in \cite{MR0595165}):
        For every $\varepsilon > 0$,
        \begin{equation}
            s_b^{-2} \, \EE\big[Z_{1,b}^2 \ind_{\{|Z_{1,b}| > \varepsilon  n^{1/2} s_b\}}\big] \longrightarrow 0, \quad \text{as } n\to \infty,
        \end{equation}
        where $s_b^2 \leqdef \EE[Z_{1,b}^2]$ and $b = b(n)\to 0$.
        But this follows from the fact that $|Z_{1,b}| \leq 2$ for all $b > 0$, and $s_b = (n\hspace{0.3mm} \VV(\hat{F}_{n,b}^{\mathrm{IGau}}))^{1/2}\to F(x) (1 - F(x))$ as $n\to \infty$ by Lemma~\ref{lem:bias.variance.IGau.kernel}.
    \end{proof}

\section{Proof of the results for the RIG kernel}\label{sec:proof.results.RIG.kernel}

    \begin{proof}[Proof of Lemma~\ref{lem:bias.variance.RIG.kernel}]
        If $T$ denotes a random variable with the density
        \begin{equation}
            k_{\mathrm{RIG}}(t \nvert \mu, \lambda) = \sqrt{\frac{\lambda}{2\pi t}} \exp\bigg(- \lambda \frac{(1 - \mu t)^2}{2 \mu^2 t}\bigg), \quad \text{with } (\mu,\lambda) = (x^{-1} (1 - b)^{-1}, x^{-1} b^{-1}),
        \end{equation}
        then integration by parts yields
        \begin{align}
            \EE[\hat{F}_{n,b}^{\mathrm{RIG}}(x)] - F(x)
            &= \EE[F(T)] - F(x) \notag \\
            &= f(x) \cdot \EE[T - x] + \frac{1}{2} f'(x) \cdot \EE[(T - x)^2] + \oo_x\big(\EE[(T - x)^2]\big) \notag \\
            &= f(x) \cdot 0 + \frac{1}{2} f'(x) \cdot x^2 b (1 + b) + \oo_x(b) \notag \\[-0.5mm]
            &= b \cdot \frac{1}{2} x^2 f'(x) + \oo_x(b).
        \end{align}
        Now, we want to compute the expression for the variance.
        Let $S$ be a random variable with density $t\mapsto 2 k_{\mathrm{RIG}}(t \nvert x^{-1} (1 - b)^{-1}, x^{-1} b^{-1}) \overline{K}_{\mathrm{RIG}}(t \nvert x^{-1} (1 - b)^{-1}, x^{-1} b^{-1})$ and note that $\min\{T_1,T_2\}$, which can also be written as $\frac{1}{2} (T_1 + T_2) - \frac{1}{2} |T_1 - T_2|$, has that particular distribution if $T_1,T_2\sim \mathrm{RIG}(x^{-1} (1 - b)^{-1}, x^{-1} b^{-1})$ are independent.
        Then, integration by parts together with the fact that $\EE[T_1] = \EE[T_2] = x$ yield, for any given $x\in (0,\infty)$,
        \begin{align}
            &\EE\big[\overline{K}_{\mathrm{RIG}}^2(X_1, x^{-1} (1 - b)^{-1}, x^{-1} b^{-1})\big] = \EE[F(S)] \notag \\
            &\qquad= F(x) + f(x) \cdot \EE[S - x] + \OO_x\big(\EE[(S - x)^2]\big) \notag \\
            &\qquad= F(x) + f(x) \cdot \Big\{- \frac{1}{2} \EE[|T_1 - T_2|]\Big\} + \OO_x(b) \notag \\[-1mm]
            &\qquad= F(x) - b^{1/2} \cdot \frac{f(x)}{2} \Big[\lim_{b\to 0} b^{-1/2} \, \EE[|T_1 - T_2|]\Big] + \OO_x(b).
        \end{align}
        so that
        \begin{align}
            \VV(\hat{F}_{n,b}^{\mathrm{RIG}}(x))
            &= n^{-1} \EE\big[\overline{K}_{\mathrm{RIG}}^2(X_1, x^{-1} (1 - b)^{-1}, x^{-1} b^{-1})\big] - n^{-1} \big(\EE[\hat{F}_{n,b}^{\mathrm{RIG}}(x)]\big)^2 \notag \\[1.5mm]
            &= n^{-1} F(x) (1 - F(x)) \notag \\[0mm]
            &\quad- n^{-1} b^{1/2} \cdot \frac{f(x)}{2} \Big[\lim_{b\to 0} b^{-1/2} \, \EE[|T_1 - T_2|]\Big] + \OO_x(n^{-1} b).
        \end{align}
        This ends the proof.
    \end{proof}

    \begin{proof}[Proof of Proposition~\ref{prop:MISE.RIG.kernel}]
        Note that $\hat{F}_{n,b}^{\mathrm{RIG}}(x) - \EE[\hat{F}_{n,b}^{\mathrm{RIG}}(x)] = \frac{1}{n} \sum_{i=1}^n Z_{i,b}$ where
        \begin{equation}
            \begin{aligned}
                Z_{i,b}
                &\leqdef \overline{K}_{\mathrm{RIG}}(X_i \nvert x^{-1} (1 - b)^{-1}, x^{-1} b^{-1}) \\[1mm]
                &\quad- \EE[\overline{K}_{\mathrm{RIG}}(X_i \nvert x^{-1} (1 - b)^{-1}, x^{-1} b^{-1})], \quad 1 \leq i \leq n,
            \end{aligned}
        \end{equation}
        are i.i.d.\ and centered random variables.
        It suffices to show the following Lindeberg condition for double arrays (see, e.g., Section 1.9.3 in \cite{MR0595165}):
        For every $\varepsilon > 0$,
        \begin{equation}
            s_b^{-2} \, \EE\big[Z_{1,b}^2 \ind_{\{|Z_{1,b}| > \varepsilon  n^{1/2} s_b\}}\big] \longrightarrow 0, \quad \text{as } n\to \infty,
        \end{equation}
        where $s_b^2 \leqdef \EE[Z_{1,b}^2]$ and $b = b(n)\to 0$.
        But this follows from the fact that $|Z_{1,b}| \leq 2$ for all $b > 0$, and $s_b = (n\hspace{0.3mm} \VV(\hat{F}_{n,b}^{\mathrm{RIG}}))^{1/2}\to F(x) (1 - F(x))$ as $n\to \infty$ by Lemma~\ref{lem:bias.variance.RIG.kernel}.
    \end{proof}

\section{Technical lemmas}\label{sec:tech.lemmas}

    The lemma below computes the first two moments for the minimum of two i.i.d.\ random variables with a Gamma distribution.
    The proof is a slight generalization of the answer provided by Felix Marin in the following \href{https://math.stackexchange.com/questions/3910094/how-to-compute-this-double-integral-involving-the-gamma-function}{MathStackExchange post}.%
    \footnote{\tiny \url{https://math.stackexchange.com/questions/3910094/how-to-compute-this-double-integral-involving-the-gamma-function}}

    \begin{lemma}\label{lem:tech.var.G.prelim}
        Let $X, Y\stackrel{\mathrm{i.i.d.}}{\sim} \mathrm{Gamma}\hspace{0.2mm}(\alpha, \theta)$, then
        \begin{equation}
            \EE[(\min\{X,Y\})^j] = \frac{\theta^j \Gamma(\alpha + j)}{\Gamma(\alpha)} - \frac{j \theta^j}{\sqrt{\pi}} \cdot \frac{\Gamma(\alpha + j - 1/2)}{\Gamma(\alpha)}, \quad j\in \{1,2\},
        \end{equation}
        where $\Gamma(\alpha) \leqdef \int_0^{\infty} t^{\alpha - 1} e^{-t} \rd t$ denotes the gamma function.
        In particular, for all $x\in \R$,
        \begin{align}
            \EE[(\min\{X,Y\} - x)]
            &= \theta \alpha - \frac{\theta}{\sqrt{\pi}} \cdot \frac{\Gamma(\alpha + 1/2)}{\Gamma(\alpha)} - x, \\
            \EE[(\min\{X,Y\} - x)^2]
            &= \theta^2 \alpha (\alpha + 1) - \frac{2 \theta^2}{\sqrt{\pi}} \cdot \frac{\Gamma(\alpha + 3/2)}{\Gamma(\alpha)} \notag \\
            &\quad- 2 x \left[\theta \alpha - \frac{\theta}{\sqrt{\pi}} \cdot \frac{\Gamma(\alpha + 1/2)}{\Gamma(\alpha)}\right] + x^2.
        \end{align}
    \end{lemma}

    \begin{proof}
        Assume throughout the proof that $j\in \{1,2\}$.
        By the simple change of variables $(u,v) = (x/\theta,y/\theta)$, we have
        \begin{align}\label{eq:complex.integral.Gamma.kernel.eq.1}
            \EE[(\min\{X,Y\})^j]
            &= 2 \int_0^{\infty} \int_0^{\infty} y^j \frac{(xy)^{\alpha - 1} e^{-(x+y)/\theta}}{\theta^{2\alpha} \Gamma^2(\alpha)} \, \ind_{[0,\infty)}(x - y) \rd x \rd y \notag \\
            &= \frac{2 \theta^j}{\Gamma^2(\alpha)} \int_0^{\infty} \int_0^{\infty} u^{\alpha - 1} e^{-u} v^{\alpha + j - 1} e^{-v} \, \ind_{[0,\infty)}(u - v) \rd u \rd v.
        \end{align}
        By the integral representation of the Heaviside function
        \begin{equation}
            \ind_{[0,\infty)}(x) = \lim_{\e\to 0^+} \frac{1}{2\pi \ii} \int_{-\infty}^{\infty} \frac{1}{\tau - \ii \e} e^{\ii x \tau} \rd \tau,
        \end{equation}
        the above is
        \begin{align}\label{eq:complex.integral.Gamma.kernel.eq.2}
            &= \frac{2 \theta^j}{\Gamma^2(\alpha)} \lim_{\e\to 0^+} \frac{1}{2\pi \ii} \int_{-\infty}^{\infty} \frac{1}{\tau - \ii \e} \underbrace{\int_0^{\infty} u^{\alpha - 1} e^{-(1 - \ii \tau) u} \rd u}_{=~(1 - \ii \tau)^{-\alpha} \Gamma(\alpha)} \underbrace{\int_0^{\infty} v^{\alpha + j - 1} e^{-(1 + \ii \tau) v} \rd v}_{=~(1 + \ii \tau)^{-\alpha - j} \Gamma(\alpha + j)} \rd \tau \notag \\
            &= \frac{2 \theta^j \Gamma(\alpha + j)}{\Gamma(\alpha)} \lim_{\e\to 0^+} \frac{1}{2\pi \ii} \int_{-\infty}^{\infty} \frac{(1 + \tau^2)^{-\alpha} (1 + \ii \tau)^{- j}}{\tau - \ii \e} \rd \tau \notag \\
            &= \frac{2 \theta^j \Gamma(\alpha + j)}{\Gamma(\alpha)} \left\{\hspace{-1mm}
                \begin{array}{l}
                    \mathrm{P.V.} \, \frac{1}{2 \pi \ii} \int_{-\infty}^{\infty} \frac{(1 + \tau^2)^{-\alpha} (1 + \ii \tau)^{- j}}{\tau} \rd \tau \\[1mm]
                    + \frac{1}{2 \pi \ii} \int_{-\infty}^{\infty} (1 + \tau^2)^{-\alpha} (1 + \ii \tau)^{- j} \cdot \ii \pi \delta(\tau) \rd \tau
                \end{array}
                \hspace{-1mm}\right\},
        \end{align}
        where $\delta$ denotes the Dirac delta function.
        The second term in the last brace is $1/2$ and the principal value is
        \begin{align}
            &=\frac{1}{2 \pi \ii} \int_0^{\infty} \frac{(1 + \tau^2)^{-\alpha}}{\tau} \left[\frac{1}{(1 + \ii \tau)^j} - \frac{1}{(1 - \ii \tau)^j}\right] \rd \tau \notag \\
            &= - \frac{j}{\pi} \int_0^{\infty} (1 + \tau^2)^{-\alpha-j} \rd \tau,
        \end{align}
        where we crucially used the fact that $j\in \{1,2\}$ to obtain the last equality.
        Putting all the work back in \eqref{eq:complex.integral.Gamma.kernel.eq.2}, we get
        \begin{align}\label{eq:complex.integral.Gamma.kernel.eq.4}
            \EE[(\min\{X,Y\})^j]
            &= \frac{2 \theta^j \Gamma(\alpha + j)}{\Gamma(\alpha)} \left\{- \frac{j}{\pi} \int_0^{\infty} (1 + \tau^2)^{-\alpha-j} \rd \tau + \frac{1}{2}\right\} \notag \\
            &= \frac{2 \theta^j \Gamma(\alpha + j)}{\Gamma(\alpha)} \left\{- \frac{j}{2\pi} \int_0^{\infty} t^{1/2 - 1} (1 + t)^{-\alpha-j} \rd t + \frac{1}{2}\right\}.
        \end{align}
        The remaining integral can be evaluated using Ramanujan's master theorem.
        Indeed, note that
        \begin{align}
            (1 + t)^{-\alpha-j}
            &= \sum_{k=0}^{\infty} \binom{-\alpha - j}{k} t^k = \sum_{k=0}^{\infty} \binom{\alpha + j + k - 1}{k} (-t)^k \notag \\
            &= \sum_{k=0}^{\infty} \varphi(k) \frac{(-t)^k}{k!}, \quad \text{with } \varphi(z) \leqdef \frac{\Gamma(\alpha + j + z)}{\Gamma(\alpha + j)}.
        \end{align}
        Therefore,
        \begin{equation}
            \int_0^{\infty} t^{1/2 - 1} (1 + t)^{-\alpha-j} \rd t = \Gamma(1/2) \varphi(-1/2) = \frac{\sqrt{\pi} \, \Gamma(\alpha + j - 1/2)}{\Gamma(\alpha + j)}.
        \end{equation}
        By putting this result in \eqref{eq:complex.integral.Gamma.kernel.eq.4}, we obtain
        \begin{equation}
            \EE[(\min\{X,Y\})^j] = \frac{\theta^j \Gamma(\alpha + j)}{\Gamma(\alpha)} - \frac{j \theta^j}{\sqrt{\pi}} \cdot \frac{\Gamma(\alpha + j - 1/2)}{\Gamma(\alpha)}.
        \end{equation}
        This ends the proof.
    \end{proof}

    \begin{corollary}\label{cor:tech.var.G}
        Let $X, Y\stackrel{\mathrm{i.i.d.}}{\sim} \mathrm{Gamma}\hspace{0.2mm}(b^{-1} x + 1, b)$ for some $x,b\in (0,\infty)$, then
        \begin{align}
            \EE[(\min\{X,Y\} - x)]
            &= b \Big(\frac{x}{b} + 1\Big) - \frac{b}{\sqrt{\pi}} \cdot \frac{\Gamma(\frac{x}{b} + 3/2)}{\Gamma(\frac{x}{b} + 1)} - x \notag \\[-1mm]
            &= -\sqrt{\frac{b x}{\pi}} + b + \OO_x(b^{3/2}), \\[1mm]
            \EE[(\min\{X,Y\} - x)^2]
            &= b^2 \Big(\frac{x}{b} + 1\Big) \Big(\frac{x}{b} + 2\Big) - \frac{2 b^2}{\sqrt{\pi}} \cdot \frac{\Gamma(\frac{x}{b} + 5/2)}{\Gamma(\frac{x}{b} + 1)} \notag \\[-1mm]
            &\quad- 2 x \bigg[x - \sqrt{\frac{b x}{\pi}} + b + \OO_x(b^{3/2})\bigg] + x^2 \notag \\
            &= b x + \OO_x(b^{3/2}).
        \end{align}
    \end{corollary}

    The lemma below computes the first two moments for the minimum of two i.i.d.\ random variables with an inverse Gamma distribution.

    \begin{lemma}\label{lem:tech.var.IGam.prelim}
        Let $X, Y\stackrel{\mathrm{i.i.d.}}{\sim} \mathrm{InverseGamma}\hspace{0.2mm}(\alpha, \theta)$ and assume $\alpha > 2$, then
        \begin{equation}
            \EE[(\min\{X,Y\})^j] = \frac{\theta^{-j} \Gamma(\alpha - j)}{\Gamma(\alpha)} - \frac{j \theta^{-j}}{\sqrt{\pi}} \cdot \frac{\Gamma(\alpha - j) \Gamma(\alpha - 1/2)}{\Gamma^2(\alpha)}, \quad j\in \{1,2\},
        \end{equation}
        where $\Phi$ denotes the c.d.f.\ of the standard normal distribution.
        In particular, for all $x\in \R$,
        \vspace{-3mm}
        \begin{align}
            \EE[(\min\{X,Y\} - x)]
            &= \frac{\theta^{-1}}{\alpha - 1} \bigg[1 - \frac{1}{\sqrt{\pi}} \cdot \frac{\Gamma(\alpha - 1/2)}{\Gamma(\alpha)}\bigg] - x, \\[1mm]
            \EE[(\min\{X,Y\} - x)^2]
            &= \frac{\theta^{-2}}{(\alpha - 1) (\alpha - 2)} \bigg[1 - \frac{2}{\sqrt{\pi}} \cdot \frac{\Gamma(\alpha - 1/2)}{\Gamma(\alpha)}\bigg] \notag \\
            &\quad- 2 x \frac{\theta^{-1}}{\alpha - 1} \bigg[1 - \frac{1}{\sqrt{\pi}} \cdot \frac{\Gamma(\alpha - 1/2)}{\Gamma(\alpha)}\bigg] + x^2.
        \end{align}
    \end{lemma}

    \begin{proof}
        Assume throughout the proof that $j\in \{1,2\}$.
        By the simple change of variables $(u,v) = (x^{-1}/\theta,y^{-1}/\theta)$ and the reparametrization $\widetilde{\alpha} \leqdef \alpha - j > 0$, we have
        \begin{align}
            \EE[(\min\{X,Y\})^j]
            &= 2 \int_0^{\infty} \int_0^{\infty} y^j \frac{(xy)^{-\alpha - 1} e^{-(x^{-1} + y^{-1})/\theta}}{\theta^{2\alpha} \Gamma^2(\alpha)} \, \ind_{[0,\infty)}(x - y) \rd x \rd y \notag \\
            &= \frac{2 \theta^{-j}}{\Gamma^2(\alpha)} \int_0^{\infty} \int_0^{\infty} u^{\alpha - 1} e^{-u} v^{\alpha - j - 1} e^{-v} \, \ind_{[0,\infty)}(v - u) \rd u \rd v \notag \\
            &= \frac{2 \theta^{-j}}{\Gamma^2(\alpha)} \int_0^{\infty} \int_0^{\infty} u^{\widetilde{\alpha} + j - 1} e^{-u} v^{\widetilde{\alpha} - 1} e^{-v} \, \ind_{[0,\infty)}(v - u) \rd u \rd v.
        \end{align}
        We already evaluated this double integral in the proof of Lemma~\ref{lem:tech.var.G.prelim} (with $\alpha$ instead of $\widetilde{\alpha}$).
        The above is
        \begin{equation}
            = \frac{2 \theta^{-j}}{\Gamma^2(\alpha)} \cdot \frac{\Gamma(\widetilde{\alpha})}{2} \left[\Gamma(\widetilde{\alpha} + j)- \frac{j}{2 \sqrt{\pi}} \cdot \Gamma(\widetilde{\alpha} + j - 1/2)\right].
        \end{equation}
        This ends the proof.
    \end{proof}

    \begin{corollary}\label{cor:tech.var.IGam}
        Let $X, Y\stackrel{\mathrm{i.i.d.}}{\sim} \mathrm{InverseGamma}\hspace{0.2mm}(b^{-1} + 1, x^{-1} b)$ for some $x\in (0,\infty)$ and $b\in (0,1)$, then
        \vspace{-1mm}
        \begin{align}
            \EE[(\min\{X,Y\} - x)]
            &= - \frac{x}{\sqrt{\pi}} \cdot \frac{\Gamma(b^{-1} + 1/2)}{\Gamma(b^{-1} + 1)} \notag \\[-0.5mm]
            &= - x \sqrt{\frac{b}{\pi}} + \OO_x(b^{3/2}), \\[1mm]
            \EE[(\min\{X,Y\} - x)^2]
            &= x^2 \bigg[\frac{1}{1 - b} - 1\bigg] - \frac{2 x^2}{\sqrt{\pi}} \frac{\Gamma(b^{-1} + 1/2)}{\Gamma(b^{-1} + 1)} \bigg[\frac{1}{1 - b} - 1\bigg] \notag \\[0.5mm]
            &= b x^2 + \OO_x(b^{3/2}).
        \end{align}
    \end{corollary}

    \newpage
    The lemma below computes the first two moments for the minimum of two i.i.d.\ random variables with a lognormal distribution.

    \begin{lemma}\label{lem:tech.var.LN.prelim}
        Let $X, Y\stackrel{\mathrm{i.i.d.}}{\sim} \mathrm{LogNormal}\hspace{0.2mm}(\mu,\sigma)$, then
        \begin{equation}
            \EE[(\min\{X,Y\})^a] = 2 e^{a \mu + \frac{(a \sigma)^2}{2}} \Phi\Big(- \frac{a \sigma}{\sqrt{2}}\Big), \quad a > 0,
        \end{equation}
        where $\Phi$ denotes the c.d.f.\ of the standard normal distribution.
        In particular, for all $x\in \R$,
        \begin{align*}
            &\EE[(\min\{X,Y\} - x)] = 2 e^{\mu + \frac{\sigma^2}{2}} \Phi\Big(- \frac{\sigma}{\sqrt{2}}\Big) - x, \\[1mm]
            &\EE[(\min\{X,Y\} - x)^2] = 2 e^{2 \mu + 2 \sigma^2} \Phi\big(- \sqrt{2} \, \sigma\big) - 4 x e^{\mu + \frac{\sigma^2}{2}} \Phi\Big(- \frac{\sigma}{\sqrt{2}}\Big) + x^2.
        \end{align*}
    \end{lemma}

    \begin{proof}
        With the change of variables
        \begin{equation}
            \begin{pmatrix}
                u \\
                v
            \end{pmatrix}
            = \frac{1}{\sqrt{2}}
            \begin{pmatrix}
                1 & -1 \\
                1 & 1
            \end{pmatrix}
            \begin{pmatrix}
                x \\
                y
            \end{pmatrix}, \qquad \left|\frac{\rd (u,v)}{\rd (x,y)}\right| = 1,
        \end{equation}
        we have
        \begin{align}
            \EE[(\min\{X,Y\})^a]
            &= 2 \int_{-\infty}^{\infty} \int_y^{\infty} e^{a(\mu + \sigma y)} \frac{1}{2\pi} e^{-\frac{x^2 + y^2}{2}} \rd x \rd y \notag \\
            &= 2 \int_{-\infty}^{\infty} \int_0^{\infty} e^{a(\mu + \sigma \cdot \frac{-u + v}{\sqrt{2}})} \frac{1}{2\pi} e^{-\frac{u^2 + v^2}{2}} \rd u \rd v \notag \\
            &= 2 e^{a \mu + \frac{(a \sigma)^2}{2}} \int_{-\infty}^{\infty} \int_0^{\infty} \frac{1}{\sqrt{2\pi}} e^{-\frac{(u + \frac{a \sigma}{\sqrt{2}})^2}{2}} \frac{1}{\sqrt{2\pi}} e^{-\frac{(v - \frac{a \sigma}{\sqrt{2}})^2}{2}} \rd u \rd v \notag \\
            &= 2 e^{a \mu + \frac{(a \sigma)^2}{2}} \Phi\Big(- \frac{a \sigma}{\sqrt{2}}\Big).
        \end{align}
        This ends the proof.
    \end{proof}

    \begin{corollary}\label{cor:tech.var.LN}
        Let $X, Y\stackrel{\mathrm{i.i.d.}}{\sim} \mathrm{LogNormal}\hspace{0.2mm}(\log x, \sqrt{b})$ for some $x,b\in (0,\infty)$, then
        \begin{align*}
            \EE[(\min\{X,Y\} - x)]
            &= x \Big[2 e^{b / 2} \Phi\Big(- \sqrt{\frac{b}{2}}\Big) - 1\Big] \notag \\
            &= - x \sqrt{\frac{b}{\pi}} + \frac{b x}{2} + \OO_x(b^{3/2}), \\
            \EE[(\min\{X,Y\} - x)^2]
            &= x^2 \Big[2 e^{2 b} \Phi\big(- \sqrt{2 b}\big) - 4 e^{b / 2} \Phi\Big(- \sqrt{\frac{b}{2}}\Big) + 1\Big] \notag \\
            &= b x^2 + \OO_x(b^{3/2}).
        \end{align*}
    \end{corollary}

\section{R code}\label{sec:R.code}

    The \texttt{R} code for the simulations in Section~\ref{sec:numerical.study} is available online.

\end{appendices}

\section*{Acknowledgments}

F.\ Ouimet is supported by a postdoctoral fellowship from the NSERC (PDF) and the FRQNT (B3X supplement).
We thank Benedikt Funke for reminding us of the representation $\min\{T_1,T_2\} = \frac{1}{2} (T_1 + T_2) - \frac{1}{2} |T_1 - T_2|$, which helped tightening up the MSE and MISE results in Section~\ref{sec:IGau.kernel.results} and Section~\ref{sec:RIG.kernel.results}.
This research includes computations using the computational cluster Katana supported by Research Technology Services at UNSW Sydney.

%
%

\phantomsection
\addcontentsline{toc}{chapter}{References}

\bibliographystyle{authordate1}
\bibliography{Lafaye_Ouimet_2020_cdf_estimation_bib}

\typeout{get arXiv to do 4 passes: Label(s) may have changed. Rerun}
\end{document}